\documentclass[12pt,a4paper]{article}
\usepackage{ifthen} 
\usepackage{calc} 

\usepackage[german,english]{babel} 

\usepackage[all]{xy}
\usepackage{amsmath}
\usepackage{amssymb}
\usepackage{latexsym}
\usepackage{enumerate}
\usepackage{theorem}
\usepackage{mathrsfs}
\usepackage{helvet}

\usepackage{lastpage} 
\usepackage{prettyref} 
\usepackage[pagebackref]{hyperref} 

\input xy 

\newboolean{finalversion} 
\setboolean{finalversion}{true}

\theoremheaderfont{\scshape}

\newrefformat{chap}{\hyperref[{#1}]{Chapter~\ref*{#1}}}
\newrefformat{sec}{\hyperref[{#1}]{Section~\ref*{#1}}}

\newtheorem{Notation}{Notation}
\newtheorem{Convention}{Convention}
\newtheorem{Def-Prop}{Definition-Proposition}
\newtheorem{Proof-Expl}{Proof-Explanation}
\newtheorem{theorem}{Theorem}[section]
\newrefformat{thm}{\hyperref[{#1}]{Theorem~\ref*{#1}}}
\newtheorem{definition}[theorem]{Definition}
\newrefformat{def}{\hyperref[{#1}]{Definition~\ref*{#1}}}
\newtheorem{lemma}[theorem]{Lemma}
\newrefformat{lem}{\hyperref[{#1}]{Lemma~\ref*{#1}}}
\newtheorem{proposition}[theorem]{Proposition}
\newrefformat{prop}{\hyperref[{#1}]{Proposition~\ref*{#1}}}
\newtheorem{corollary}[theorem]{Corollary}
\newrefformat{cor}{\hyperref[{#1}]{Corollary~\ref*{#1}}}
{\theorembodyfont{\upshape}
\newtheorem{examplecore}[theorem]{Example}}
\newrefformat{ex}{\hyperref[{#1}]{Example~\ref*{#1}}}

\newtheorem{remark}[theorem]{Remark}
\newrefformat{rem}{\hyperref[{#1}]{Remark~\ref*{#1}}}

\newcommand{\Spec}{\ensuremath{\operatorname{Spec}}}

\newcommand{\Mod}{\ensuremath{\operatorname{Mod}}}

\newcommand{\Hom}{\ensuremath{\operatorname{Hom}}}

\newcommand{\supp}{\ensuremath{\operatorname{supp}}}

\newcommand{\id}{\ensuremath{\operatorname{id}}}

\newcommand{\ann}{\ensuremath{\operatorname{Ann}}}

\newenvironment{proof}{\noindent\textsc{Proof:}}{\hspace*{\fill}
$\blacksquare$\par\vspace{.1cm}}

\newcommand{\mylabel}[1]{\label{#1}\ifthenelse{\boolean{finalversion}}{
  }{\marginpar{\tiny #1}}}  

\setboolean{finalversion}{true}
\begin{document}
\title{Basics of jet modules and algebraic linear differential operators}
\maketitle
\author{Stefan G\"unther}
\begin{abstract} In this paper, we collect the fundamental basic properties of jet-modules in algebraic geometry and related properties of differential operators. We claim no originality but we want to provide a reference work for own research and the research of other people.
\end{abstract}

\tableofcontents
\section{Notation and Conventions}
\begin{remark}\mylabel{rem:R1912} This is a slightly advanced introduction to the theory of jet modules in algebraic geometry. For the elementary facts see e.g. \cite{Vakil}.
\end{remark}
\begin{Convention} By $\mathbb N$\, we denote the natural numbers, by $\mathbb N_0$\, the set of nonnegative integers.
\end{Convention}
We use multi index notation: if $x_1,...,x_n$\, is a set of variables, we denote  
$$\underline{x}^{\underline{m}}: x_1^{m_1}\cdot x_2^{m_2}\ldots x_n^{m_n}\quad  \text{where}\quad  \underline{m}:=(m_1,m_2,\ldots ,m_n)$$ is a multiindex of lenght $n$. By $\mid \underline{m}\mid$\, we denote the number $m_1+...m_n$\,.
The partial derivatives of a function $f(x_1,...,x_m)$\,  in the variables $x_i$\, we denote by 
$\partial^{\mid m\mid}/\partial \underline{x}^{\underline{m}}(f(x_1,...,x_m))$\,.
\begin{Notation} Let $X\longrightarrow S$\, be a morphism of schemes. By $\Omega^{(1)}(X/S)$\, we denote the usual sheaf of K\"ahler differentials. We use this notation, because there exist higher K\"ahler differential modules $\Omega^{(n)}(X/S)$\, see \cite{Guenther}. The direct sum $\bigoplus_{n\in \mathbb N_0}\Omega^{(n)}(X/S)$\, is a graded sheaf of $\mathcal O_X$-algebras whose $\Spec_X$\, is the well known space of relative arcs (arcs in fibre direction). 
\end{Notation}
\begin{Notation} Let $k\longrightarrow A$\, be a homomorphism of commutative rings. By $I_{A/k}$\, we denote the ideal in the ring $A\otimes_kA$\, which is the kernel of the multiplication map $\mu: A\otimes_kA\longrightarrow A$\,. By $p_i: A\longrightarrow A\otimes_kA, i=1,2$\, we denote the maps $p_1: a\mapsto a\otimes 1$\, and $p_2: a\mapsto 1\otimes a$\, and likewise for the residue rings (jet algebras) $A\otimes_kA/I_{A/k}^{N+1}$\,.
\end{Notation}
\begin{Notation} Let $k\longrightarrow A$\, be a homomorphism of commutative rings and let $M$\, be an $A$-module. For each $N\in \mathbb N_0\cup\{\mathbb N\}$\, we denote the $N^{th}$ jet-module of $M$ relative to $k$\, by $\mathcal J^N(M/k)$\, which is by definition the module 
\begin{gather*}\mathcal J^N(M/k):=A\otimes_kM/I_{A/k}^{N+1}\cdot (A\otimes_kM)\quad \text{if} \quad N\in \mathbb N_0\quad \text{and}\\
 \mathcal J^{\mathbb N}(M/k):=\widehat{A\otimes_kM}^{I_{A/k}},
 \end{gather*} which is the completed module with respect to the diagonal ideal $I_{A/k}$\,.\\
 If $M=A$\,, this is a ring, which has two $A$-algebra structures in the obvious way. $\mathcal J^N(M/k)$\, is an $\mathcal J^N(A/k)$-module in a canonical way. \\
  The  derivation
$$M\longrightarrow \mathcal J^N(M/k),\quad m\mapsto \overline{1\otimes m}$$ is denoted by $d^N_{M/k}$\,.\\
Because of \prettyref{lem:L45} proven in section 2.1, it is called the universal  derivation for the $A$-module  $M$ relative to $k$\,.\\
If $M=A$\,, for $a\in A$\, we denote by $d^1a$\, the element $\overline{1\otimes a}-\overline{a\otimes 1}\in \mathcal J^N(A/k)$\,. Thus $d^N_{A/k}(a)=a +d^1(a).$\\
If $X\longrightarrow S$\, is a morphism of schemes an $\mathcal F$\, is a quasi coherent $\mathcal O_X$-module, we denote the $N^{th}$-jet module by $\mathcal J^N(\mathcal F/S)$\, with universal derivation
$$d^N_{\mathcal F/S}: \mathcal F\longrightarrow \mathcal J^N(\mathcal F/S).$$
Observe that in \cite{EGA}[EGAIV,chapter 16.3-16.4, pp.14-27], these are called the bundles of higher order principal parts.\\
 If $\mathcal F=\mathcal O_X$\, we denote 
 $$J^N(X/S):=\Spec_X\mathcal J^N(X/S)\stackrel{p_X}\longrightarrow X$$
 the associated affine bundle over $X$ which is some kind of higher order relative tangent bundle.
\end{Notation}
\section{Introduction}
The calculus of jet-modules and jet bundles in algebraic geometry is basic for understanding linear partial differential operators and for a given extension of commutative rings $k\longrightarrow A$\, and an $A$-module $M$, the $N^{th}$\, jet module $\mathcal J^N(M/k)$\, provides infinitesimal information about the $A$-module $M$. If $M=A$\,, and $A$\, is a local ring, essentially of finite type over a base field $k$,  the $N^{th}$ jet-algebra $\mathcal J^N(A/k)$\, provides further information about the singularity $(A,\mathfrak{m}, k).$\, This algebra is again a $k$-algebra, essentially of finite type over $k$, and, e.g., the Hilbert-Samuel polynom of this algebra gives higher order information about the singularity $(A,\mathfrak{m},k)$\,. So, studying jet modules can be very fruitful and in this paper we want to give basic elementary properties, e.g. generalizing properties of the classical module of K\"ahler differentials. We claim no originality but want to collect some basic facts that do not occur in the basic textbooks as \cite{EGA}[EGAIV,chapter 16.3-16.4, pp.14-27] in order to provide a reference for further own research and the research of other people. 
\section{Fundamentals of jet-modules and differential operators in algebraic geometry}
\subsection{Definition of a linear partial differential operator}
Recall that if $k$ is a field and $\mathbb A^n_k$\, is affine $n$-space over $k$, then a homogenous partial linear differential operator of order $N$, 
$$D: \bigoplus_{i=1}^m k[x_1,...,x_n]e_i\longrightarrow \bigoplus_{i=1}^mk[x_1,...,x_n]e_i$$ corresponds to a $k[x_1,...,x_n]$-linear map 
$$\widetilde{D}: (k[x_1,...,x_n][d^1x_1,...,d^1x_n]/(d^1x_1,...,d^1x_n)^{N+1})^{\oplus m}\longrightarrow k[x_1,...,x_n]^{\oplus m}$$ under the natural correspondence $\widetilde{D}\mapsto \widetilde{D}\circ (d^N_{\mathbb A^n_k/k})^{\oplus m}$\, where $d^N_{\mathbb A^N_k/k}$\, is the $N$-truncated Taylor series expansion,
$$k[x_1,...,x_n]\longrightarrow k[x_1,...,x_n][d^1x_1,...,d^1x_n]/(d^1x_1,...,d^1x_n)^{N+1},$$ sending $x_i$\, to $x_i+d^1x_i$\,.  This is a standard calculation. The $k$-algebra $k[\underline{x}][\underline{d^1x}]/(\underline{d^1x}^{N+1})$\, is the $N^{th}$ jet module $\mathcal J^N(k[\underline{x}]/k)$\, of $k[\underline{x}]/k$\, and is a $k[\underline{x}]$-algebra. The inverse limit 
\begin{gather*}\mathcal J^{\mathbb N}(k[\underline{x}]/k):=\projlim_{n\in\mathbb N}\mathcal J^N(k[\underline{x}]/k)=k[x_1,...,x_n][[d^1x_1,...,d^1x_n]]\\
\cong k[x_1,...,x_n]\widehat{\otimes}k[x_1,...,x_n]
\end{gather*} is the universal jetalgebra of $k[\underline{x}]$\, where the last expression is the tensor product completed with respect to the ideal $I_{k[\underline{x}]/k}$\, which is the kernel of the algebra multiplication map.
Over $k[\underline{x}]$ each projective module is free, so the standard definition from the $\mathcal C^{\infty}$-case carries over to the algebraic case. For $k=\mathbb R,\mathbb C$\, these are just the linear partial differential operators with polynomial coefficients. Via the above corresondence each partial differential operator corresponds to an $A$-linear map 
$$\widetilde{D}: \mathcal J^N(A^{\oplus m}/k)\longrightarrow A^{\oplus m}\,\, (A=k[\underline{x}]).$$ 
Since the formation of the jet modules commutes with localization, they give rise to a coherent sheaf on $\mathbb A^n$\, and we define the algebraic differential operators for a free $A=\Gamma(U,\mathcal O_{\mathbb A^n})$-module $A^{\oplus m}$ for $U\subset \mathbb A^n_k$\, a Zariski open subset  to be the $A$-linear homomorphisms $\Gamma(U, \mathcal J^N(k[\underline{x}]/k))^{\oplus m}\longrightarrow A^{\oplus m}$\,.\\
 In particular, if $U=\Spec k[\underline{x}]_{f}$\, for some polynomial $f$, this definition gives linear partial differential operators with rational function coefficients $\frac{g}{f^n}$\, .\\
\subsection{Supplements to the calculus of jet modules}
\begin{definition}\mylabel{def:D1712} Let $k\longrightarrow A$\, be a homomorphism of commutative rings, let $N\in \mathbb N_0\cup \{\mathbb N\}$, $M$\, be an $A$-module   and $Q$\, be a $\mathcal J^N(A/k)$-module. A filtered derivation $t_M: M\longrightarrow Q$\,  is an $A$-linear homomorphism where $Q$ is given the $A$-module structure via the second factor $A\longrightarrow A\otimes_kA/I_{A/k}^{N+1}$\,.\\  The notation filtered derivation, introduced in \cite{Guenther}, is used because for each $a\in A$, $m\in M$, $t(am)-am\in I_{A/k}\cdot Q$\,.
\end{definition}
\begin{lemma}\mylabel{lem:L45} Let $A\longrightarrow B$\, be a homomorphism of rings, $M$ be a $B$-module, and $Q$ be a $\mathcal J^N(B/A)$\,-module and $t: M\longrightarrow Q$\, be a $B$-linear map with respect to the second $B$-module structure on $\mathcal J^N(B/A)$, i.e., a filtered derivation. Then, there is a unique homomorphism of $\mathcal J^N(B/A)$\,-modules 
$$\phi:\mathcal J^N(M/A)\longrightarrow Q\quad \text{such that}\quad  t=\phi\circ d^N_{M/A}.$$
\end{lemma}
\begin{proof} We have $\mathcal J^N(B/A)=B\otimes_AB/\mathcal I_{B/A}^{N+1}$\, and natural homomorphisms $p_1,p_2: B\longrightarrow \mathcal J^N(B/A)$\,. Then, by definition 
$$\mathcal J^N(M/A)=M\otimes_{B,p_2}\mathcal J^N(B/A).$$ The statement reduces to the easy fact, that, given a homomorphism of rings $k\longrightarrow l,$ given  a $k$-module $M_k$ and  an $l$-module $M_l$, and a $k$-linear homomorphism $M_k\longrightarrow M_l$\,, there is a unique $l$-linear homomorphism $M_k\otimes_{k}l\longrightarrow M_l$\, which follows by the adjunction of restriction and extension of scalars.
\end{proof}
\subsection{Fundamental properties of the jet-modules}
Recall e.g. from \cite{Guenther}, that if $k\longrightarrow A$\, is a $k$-algebra with presentation 
$$A\cong k[x_i| i\in I]/(f_j| j\in J),$$ then the $N^{th}$ jet algebra $\mathcal J^N(A/k)$
possesses the presentation 
$$\mathcal J^N(A/k):= k[x_i, d^1x_i|i\in I]/ (I_{A/k}^{N+1}+(f_j, d^1f_j| j\in J)).$$
This can also be taken as the definition of the $N^{th}$-jet-algebra. One then has to show that this is independend of the choosen presentation for $A$.
We have the following easy consequences.
\begin{lemma}\mylabel{lem:L19125}(Base change 0) Let $k\longrightarrow A$\, and $k\longrightarrow A'$\, be homomorphisms of commutative rings. Then for each $N\in \mathbb N_0\cup\{\mathbb N\}$\,, there is an isomorphism
$$\mathcal J^N(A\otimes_kA'/A')\stackrel{\cong}\longrightarrow \mathcal J^N(A/k)\otimes_AA'.$$
\end{lemma}
\begin{proof} This follows from the fact, that if $A=k[x_i|i\in I]/(f_j|j\in J)$\, is a presentation of $A/k$, then 
$$A\otimes_kA'=A'[x_i|i\in I]/(f_j|j\in J)$$ is a presentation for $A\otimes_kA'/A'$\, and the corresponding presentation of the jet-algebras.\\
This isomorphism can be made canonical by observing that 
$$A\otimes_kA'\stackrel{d^N_{A/k}\otimes_k\id_{A'}}\longrightarrow \mathcal J^N(A/k)\otimes_kA'\longrightarrow \mathcal J^N(A/k)\otimes_AA'$$
is a filtered module derivation which by the universal property of the jet-modules induces a canonical
isomorphism
$\mathcal J^N(A\otimes_kA'/A')\cong \mathcal J^N(A/k)\otimes_AA'.$
\end{proof}
\begin{corollary}\mylabel{cor:C19126} If $A$ is a finitely generated $k$-algebra, then for each $N\in \mathbb N_0$\,, $\mathcal J^N(A/k)$\, is a finitely generated $A$-algebra.
\end{corollary}
\begin{proof}
\end{proof} 
\begin{proposition}\mylabel{prop:P20124} Let $k\longrightarrow A$\, be a homomorphism of noetherian rings that makes $A$ a smooth, finite type $k$-algebra.  Let $M$\, be a finitely generated projective $A$-module. Then for each $N\in \mathbb N_0$\,, $\mathcal J^N(M/k)$\, is a projective $A$-module.
\end{proposition}
\begin{proof} First, consider the case $M=A$\,. We prove the result by induction on $N\in \mathbb N_0$\,. For $N=0$, $\mathcal J^0(A/k)\cong A$\, the result is trivially true.  From the jet-bundle exact sequence  $j^N_{A/k}$,
$$0\longrightarrow I_{A/k}^N/I_{A/k}^{N+1}\longrightarrow \mathcal J^N(A/k)\longrightarrow \mathcal J^{N-1}(A/k)\longrightarrow 0,$$
suppose we know that $\mathcal J^{N-1}(A/k)$\, is a projective $A$-module. Since $A$ is a smooth $k$-algebra, the diagonal ideal $I_{A/k}$\, corresponds to the regular embedding
$$X=\Spec A\hookrightarrow X\times_{\Spec k}X.$$
In this case, 
$$I_{A/k}^N/I_{A/k}^{N+1}\cong (I_{A/k}/I_{A/k}^2)^{\otimes^s N}\cong \Omega^{(1)}(A/k)^{\otimes^s N}$$ which is well known to be a projective $A$-module. The induction step is then complete by observing that an extension of projective $A$-modules is again a projective $A$-module.\\
For the general case, if $M$ is a free $A$-module, the claim is true since taking jet-modules commute with taking direct sums. \\
For arbitrary projective $M$, since by \prettyref{cor:C19127} taking jet modules commute with Zariski localizations, choosing a Zariski-open covering $\Spec A=\bigcup_{i\in I}\Spec A_i$\, such that $M$ is free on $\Spec A_i$\,, we know that $\mathcal J^N(M/k)$\, is Zariski-locally a projective $A$-module, hence a projective $A$-module.
\end{proof} 
We now prove some fundamental properties.  
\begin{proposition}\mylabel{prop:POkt2916}(Exterior products I) Let $k\longrightarrow A$\, and $k\longrightarrow B$\, be homomorphisms of commutative rings. Then, there is a canonical isomorphism
$$\gamma_{A,B}: \mathcal J^{\mathbb N}(A/k)\otimes_k\mathcal J^{\mathbb N}(B/k)\cong \mathcal J^{\mathbb N}(A\otimes_kB/k).$$
\end{proposition}
\begin{proof} This follows from the explicite presentations of the jet modules for a given presentation of $A$ and $B$, respectively (see \cite{Guenther}[chapter 6.5, pp. 101-119]).\\
Choose $N\in \mathbb N_0$\,.
 If 
 $$A=k[x_i| i\in I]/(f_j| j\in J)\quad  \text{and}\quad  B=k[y_k| k\in K]/(g_l| l\in L)$$
 are presentations for $A$ and $B$, then 
$$A\otimes_kB= k[x_i, y_k\mid i\in I, k\in K]/(f_j, g_l\mid j\in J, l\in L)$$ is a presentation for $A\otimes_kB$\, 
and 
\begin{gather*}
\mathcal J^N(k[x_i\mid i\in I]/(f_j\mid j\in J)/k)\otimes_k\mathcal J^N(k[y_k\mid k\in K]/(g_l\mid l\in L)/k)
\cong \\
k[x_i, d^1x_i| i\in I]/((f_j, d^1f_j\mid j\in J)+ I_{A/k}^{N+1})\\
\otimes_k k[y_k, d^1y_k| k\in K]/((g_l, d^1g_l\mid l\in L)+I_{B/k}^{N+1})\quad \text{and} \\
\mathcal J^N(A\otimes_kB/k)\cong \\
k[x_i, y_k, d^1x_i, d^1y_k| i\in I, k\in K]/((f_j, d^1f_j, g_l, d^1g_l| k\in K, l\in L)+I_{A\otimes_kB/k}^{N+1})
\end{gather*}
We have the identity $I_{A\otimes_kB/k}=(I_{A/k}+I_{B/k})$\,  which follows from the well known fact that the diagonal ideal is generated by all $1\otimes a-a\otimes 1=d^1a$\,. Thus $I_{A/k}$\, is generated by all $1\otimes x_i-x_i\otimes 1$, $I_{B/k}$\, is generated by all $1\otimes y_j-y_j\otimes 1$.
Obviously, we have an inclusion 
$$(I_{A/k}^{N+1}+I_{B/k}^{N+1})\cdot (A\otimes_kB)\subseteq I_{A\otimes_kB}^{N+1}.$$
Conversely, there is an inlcusion 
$$I_{A\otimes_kB/k}^{2N+1}\subseteq (I_{A/k}^{N+1}+I_{B/k}^{N+1})\cdot (A\otimes_kB)$$ because $2N+1$-fold products of elements $x_i\otimes 1-1\otimes x_i$\, and $y_j\otimes 1-1\otimes y_j$\, must either contain an $N+1$ fold product of the $x_i\otimes 1-1\otimes x_i$\, or an $N+1$ fold product of the $y_j\otimes 1-1\otimes y_j$\,.\\
It follows, taking the projective limit over the natural homomorphisms
$$\gamma_{A,B}^N:\mathcal J^N(A/k)\otimes_k\mathcal J^N(B/k)\longrightarrow \mathcal J^N(A\otimes_kB/k),$$
we get an isomorphism
$$\gamma_{A,B}:\mathcal J^{\mathbb N}(A/k)\otimes_k\mathcal J^{\mathbb N}(B/k)\stackrel{\cong}\longrightarrow\mathcal J^{\mathbb N}(A\otimes_kB/k).$$
\end{proof}
\begin{lemma}\mylabel{lem:L19121} Let $k\longrightarrow A$\, be a homomorphism of commutative rings. Then the functor 
$$\mathcal J^N(-/k): (A-\Mod)\longrightarrow (\mathcal J^N(A/k)-\Mod),\quad M\mapsto \mathcal J^N(M/k)$$
is right exact.
\end{lemma}
\begin{proof} This follows from the functorial isomorphism $\mathcal J^N(M/k)\cong M\otimes_{A,p_2}\mathcal J^N(A/k)$ and the right exactness of the tensor product.
\end{proof}
\begin{corollary}\mylabel{cor:C19123} With the previous notation, if $$A^{\oplus J}\stackrel{\phi}\longrightarrow A^{\oplus I}\twoheadrightarrow M$$ is a presentation for the $A$-module $M$, with $\phi$\, given by the matrix $(a_{ij})$, then $\mathcal J^N(M/k)$
is given by the presentation
$$\mathcal J^N(A/k)^{\oplus J}\stackrel{\mathcal J^N(\phi/k)}\longrightarrow \mathcal J^N(A/k)^{\oplus I}\twoheadrightarrow \mathcal J^N(M/k),$$
where $\mathcal J^N(\phi/k)$\, is given by the matrix $(a_{ij}+d^1a_{ij})$\,.
\end{corollary}
\begin{proof} This follows from the easy fact, that the functor $\mathcal J^N(-/k)$\, commutes with direct sums.
\end{proof} 
\begin{proposition}\mylabel{prop:POkt298} (Exterior Products II) Let $k\longrightarrow A$\, and $k\longrightarrow B$\, be homomorphisms of commutative rings. Let $M$ be an $A$-module and $N$ be a $B$-module.
Then, there is a canonical isomorphism 
$$\gamma_{M,N}:\mathcal J^{\mathbb N}(M/k)\otimes_k\mathcal J^{\mathbb N}(N/k)\cong \mathcal J^{\mathbb N}(M\otimes_kN/k).$$
\end{proposition}
\begin{proof} The standard arguement shows that there is a canonical transformation of bi-functors
$$\gamma_{M,N}: \mathcal J^{\mathbb N}(M\otimes_kN/k)\longrightarrow \mathcal J^{\mathbb N}(M/k)\otimes_k\mathcal J^{\mathbb N}(N/k)$$
that comes from  the fact the the tensor product $d^{\mathbb N}_{M/k}\otimes_kd^{\mathbb N}_{N/k}$\, is a filtered derivation  and the universal representing property of the jet-modules. If $M$ and $N$ are free $A$- and $B$-modules respectively,  the fact that $\gamma_{M,N}$\, is an isomorphism,  follows from (Exterior products I) and the fact that the jet-modules commute with direct sums. Then both sides are right exact in each variable $M,N$\, and choosing free presentations of $M$ and $N$\, respectively, the claim follows by the five lemma.
\end{proof}
\begin{proposition}\mylabel{prop:P2913} (Base change I) Let $A\longrightarrow B,$\, $A\longrightarrow A'$\,  be  homomorphisms of commutative rings and $M$ be a $B$-module. Then,  for each $N\in \mathbb N_0\cup\{\mathbb N\}$\,, there is a canonical isomorphism 
$$\beta_M:\mathcal J^N(M\otimes_AA'/A')\cong \mathcal J^N(M/A)\otimes_AA'.$$
\end{proposition}
\begin{proof} Both sides are  right exact functors from $B-\mbox{Mod}$\, to $\mathcal J^N(B\otimes_AA')-\Mod$\,. This follows from (Base change 0)( see \prettyref{lem:L19125}).\\
There is an $A'$-linear map 
$$t^N: d^N_{M/A}\otimes_A\mbox{Id}_{A'}: M\otimes_AA'\longrightarrow \mathcal J^N(M/A)\otimes_AA'$$ which is a filtered module-derivation. If $I_{B/A}$\, is the diagonal ideal, we have that 
\begin{gather*}t^N(b\cdot m\otimes a')-b\cdot m\otimes a'=d^N_M(bm)\otimes a'-bm\otimes a'\\
\in I_{B/A}\cdot \mathcal J^N(M/A)\otimes_AA'
\subseteq I_{B\otimes_AAÄ/A}\cdot (\mathcal J^N(M/A)\otimes_AA'),
\end{gather*}
which is the definition of a  filtered module derivation. By the universal property of the jet modules, there is a functorial homomorphism (natural transformation)
$$\beta_M: \mathcal J^N(M\otimes_AA'/A')\longrightarrow \mathcal J^N(M/A)\otimes_AA'$$
of functors from $B-\mbox{Mod}$\, to $\mathcal J^N(B\otimes_AA')-\text{Mod}$\,.  If $M\cong B$, by (Base change 0) (see \prettyref{lem:L19125}), there is an isomorphism 
$$\beta_B: \mathcal J^N(B\otimes_AA')\cong \mathcal J^N(B/A)\otimes_AA'.$$
Since the jet-modules commute with direct sums, $\beta_M$\, is an isomorphism for each free $B$-module $B^{\oplus I}$\,. If $M$\, is arbitrary, choose a presentation $$B^{\oplus I}\longrightarrow B^{\oplus J}\longrightarrow M\longrightarrow 0$$
 Since $\beta_{B^{\oplus I}}$\, and $\beta_{B^{\oplus J}}$\, are isomorphisms, $\beta_M$\, is an isomorphism by the five-lemma.
\end{proof}
\begin{proposition}\mylabel{prop:POkt2914}(Base change II) Let $A\longrightarrow B$\, be a homomorphism of rings, $M$ be a $B$-module and  $N$\, be an $A$-module. Then, there is a canonical isomorphism 
$$\alpha_N:\mathcal J^N(M/A)\otimes_AN\cong \mathcal J^N(M/A)\otimes_AN.$$
\end{proposition}
\begin{proof}  Fixing the $B$-module $M$,  both sides  can be considered as functors from $A-\mbox{Mod}$\, to $\mathcal J^N(B/A)-\mbox{Mod}$\,. For each $A$-module $N$, there is an $A$-linear map 
$$M\otimes_AN\stackrel{d^N_M\otimes_A\text{Id}_N}\longrightarrow \mathcal J^N(M/A)\otimes_AN.$$ This is a filtered module derivation, i.e., if $I_{B/A}$\, is the diagonal ideal, then 
\begin{gather*}d^N(b\cdot (m\otimes n))-b\cdot (m\otimes n)=d^N_{M/k}(b\cdot m)\otimes n-b\cdot m\otimes n \in \\I_{B/A}\cdot (\mathcal J^N(M/A)\otimes_AN)
\subseteq I_{B\otimes_AA'}\cdot (\mathcal J^N(M/A)\otimes_AN,
\end{gather*}
because $d^N_{M/A}$\, is a module derivation. By the representing property of the jet-modules, there is a unique homomorphism of $\mathcal J^N(B/A)$-modules
$$\alpha_N: \mathcal J^N(M\otimes_AN/A)\longrightarrow \mathcal J^N(M/A)\otimes_AN.$$
This homomorphism is in fact a natural transformation of functors from $A-\text{Mod}$\, to $\mathcal J^N(B/A)-\mbox{Mod}.$ Both functors are right exact functors. 
If $N=A^{\oplus I}$\, is a free $A$-module, then both sides are isomorphic to $\mathcal J^N(M/A)^{\oplus I}$\, because the jet-modules commute with direct sums.\\
In the general case, choose a presentation 
$$A^{\oplus I}\longrightarrow A^{\oplus J}\longrightarrow N\longrightarrow 0.$$ We know that $\alpha_{A^{\oplus I}}$\, and $\alpha_{A^{\oplus J}}$\, are isomorphisms, so by the five lemma , it follows that $\alpha_N$\, is an isomorphism.
\end{proof} 
\begin{proposition}\mylabel{prop:PNov4}(Tensor Products)\\
 Let $k\longrightarrow A$\, be a homomorphism of commutative rings and $M,N$\, be two $A$-modules. Then, for each $N\in \mathbb N_0\cup\{\mathbb N\}$\, , there is a canonical functorial isomorphism
$$\theta_{M,N}: \mathcal J^N(M\otimes_AN/k)\stackrel{\cong}\longrightarrow \mathcal J^N(M/k)\otimes_{\mathcal J^N(A/k)}\mathcal J^N(N/k),$$
in the sense that both sides are bi-functors to $\mathcal J^N(A/k)-\Mod$\, and $\theta_{M,N}$\, is a natural transformation of bifunctors that is for each object $(M,N)$\, an isomorphism.
\end{proposition}
\begin{proof} The arguement is standard.
There is a canonical homomorphism
\begin{gather*}t^N_{M\otimes_AN}:\,\,M\otimes_AN\stackrel{d^K_M\otimes_Ad^K_N}\longrightarrow \mathcal J^N(M/k)^{(2)}\otimes_A\mathcal J^N(N/k)^{(2)}\\
\twoheadrightarrow \mathcal J^N(M/k)\otimes_{\mathcal J^N(A/k)}\mathcal J^N(N/k),
\end{gather*}
where the superscript $(-)^{(2)}$\, indicates that the jet-modules are considered with respect to the second $A$-module structure.
 By the universal property of $\mathcal J^N(M\otimes_AN/k)$, there is a unique homomorphism of $\mathcal J^N(A/k)$-modules
$$\theta_{M,N}: \mathcal J^N(M\otimes_AN/k)\longrightarrow \mathcal J^N(M/k)\otimes_{\mathcal J^N(A/k)}\mathcal J^N(N/k).$$
That this is a natural transformation of bi-functors follows from the uniqueness of $\theta_{M,N}$.\\
Now, if $M= A^{\oplus I}$\, is free, then $\theta_{M,N}$\, is an isomorphism (both sides are isomorphic to 
$\bigoplus_{i\in I}\mathcal J^N(N/k)$\,.  Furthermore, both sides are right exact functors in the $M$-variable for fixed $N$, so the result follows by choosing a presentation $A^{\oplus I}\longrightarrow A^{\oplus J}\twoheadrightarrow M.$
\end{proof}
Combining (Base change I) and (Base change II) we get
\begin{proposition}\mylabel{prop:P2012}(Base change III) Let $A\longrightarrow B$\, and $A\longrightarrow A'$\, be homomorphisms of commutative rings, $M$\, be a $B$-module and $N$\, be an $A'$-module. Then, for each $N\in \mathbb N_0\cup\{\mathbb N\},$ functorial in $N$, there are isomorphisms of  $\mathcal J^N(B\otimes_AA'/A')$-modules
$$\alpha_{M,N}: \mathcal J^N(M\otimes_AN/A')\stackrel{\cong}\longrightarrow \mathcal J^N(M/A)\otimes_AN.$$
\end{proposition}
\begin{proof} The arguement is now standard.  We fix the $B$-module $M$\,. One checks that 
$$d^N_{M/A}\otimes_A\id_{N}: M\otimes_AN\longrightarrow \mathcal J^N(M/A)\otimes_AN$$
is a filtered module dervation  relative to $A'$, giving rise to 
a functorial homomorphism of $\mathcal J^N(B\otimes_AA')$-modules
$$\alpha_N: \mathcal J^N(M\otimes_AN/A')\stackrel{\cong}\longrightarrow \mathcal J^N(M/A)\otimes_AN.$$
By (Base change II), $\alpha_{A'}$\, is an isomorphism. It then follows $\alpha_N$\, is an isomorphism for a free $A'$-module $N$, and $\alpha_N$\, is then an isomorphism for each $N$ by taking free presentations and application of the five-lemma.
\end{proof}  
\begin{lemma}\mylabel{lem:LOkt2920} Let $A\longrightarrow B$\, be a  smooth homomorphism  (of finite type) of  noetherian rings and $M$\, be a  projective $B$-module. If $$0\longrightarrow N_1\longrightarrow N\longrightarrow N_2\longrightarrow 0$$
is an exact sequence of $A$-modules, then for each $N\in \mathbb N_0\cup\{\mathbb N\}$, there is an exact sequence of $\mathcal J^N(B/A)$\, modules
$$(*)_M: 0\longrightarrow \mathcal J^N(M\otimes_AN_1/A)\longrightarrow \mathcal J^N(M\otimes_AN/A)\longrightarrow \mathcal J^N(M\otimes_AN_2/A)\longrightarrow 0.$$
Furhtermore, the exact sequence $(*)_M$\, is functorial in $M$\, i.e., $(*)_M$\, is a functor from $A-\Mod$\, to the category of exact sequences in $\mathcal J^N(B/A)-Mod$\,.
\end{lemma}
\begin{proof} Follows from (Base change III) and the functor properties of the jet-modules and the fact, that in this case $\mathcal J^N(M/A)$\, is a projective $B$-module, hence a flat $A$-module. (see \prettyref{prop:P20124})
\end{proof}
\begin{remark}\mylabel{prop:P20125} We have only proved the fundamental properties of the jet-modules  (tensor products, base change ...) for $N\in \mathbb N_0$\,. But the result for $N=\mathbb N$\, follows by taking projective limits of the isomorphisms obtained for $N\in \mathbb N_0$.
\end{remark}
Because of lack of reference, we want to prove the following inocuous generalization of the formal inverse function theorem.
\begin{lemma} Let $A$\, be a noethrian ring and let  $C:= A[[y_1,...,y_n]]/(f_1,...,f_m)$\, with $m\geq n$\, be a formal power series ring such that some $(n\times n)$-minor has a determinant which is a unit in $A$.
Then $C\cong K$\,.
\end{lemma}
\begin{proof} Without loss of generality, let this be the left upper most minor . But then, in the power series ring $A[[d^1x_1,...,d^1x_n]]$\,, by the formal inverse function, theorem $d^1f_1,...,d^1f_n$\, are formal coordinates and 
$$A[[d^1x_1,...,d^1x_n]]/(d^1f_1,...,d^1f_n)\cong A$$ and a fortiori $A[[d^1x_1,...,d^1x_n]]/(d^1f_1,...,d^1f_m)\cong A$\,. The proof in \cite{Bochner} in the introductory chapter given for the case where $A$\, is a field carries over verbatim. One has to develop $d^1x_i$\, into a formal power series in the  $d^1f_j$,
$$d^1x_i=\sum_{j=1}^nb_{ji}d^1x_i+\sum_{J}(\underline{d^1f})^J.$$
 By the Cramer rule one determines the $b_{ji}$\, and then, inductively one determines the $b_J, \mid J\mid \geq 2$\,.
\end{proof}
\begin{lemma}\mylabel{lem:L30111}( etale invariance of the jet modules). Let $k\longrightarrow A\longrightarrow B$\, be homomorphism  of finite type of noetherian rings with $A\longrightarrow B$\, being etale. Let $M$\, be an $A$-module. Then, there is a canonical isomorphism
$$\alpha_M: \mathcal J^N(M\otimes_AB/k)\stackrel{\cong}\longrightarrow \mathcal J^N(M/A)\otimes_AB$$
which is a natural transformation of right exact functors from $(A-\Mod)$\, to $\mathcal J^N(B/k)-\Mod$\,.
\end{lemma}
\begin{proof} We first treat the case $M=A$\,. It suffices to show the claim for the full jet module. For finite $N\in \mathbb N_0$\,, the result follows by taking truncations.\\  We choose a presentation $B=A[x_1,...,x_n]/(f_1,..., f_m)$\,. By elementary dimension theory, we must have $m\geq n$\,. We then have 
$$\mathcal J^{\mathbb N}(B/k)=\mathcal J^{\mathbb N}(A/k)\otimes_AB\otimes_B B[[d^1x_1,...,d^1x_n]]/(d^1f_1,...,d^1f_m),$$
by choosing an appropriate presentation of $A/k$. So it suffices to show, that 
$$B[[d^1x_1,...,d^1x_n]]/(d^1f_1,...,d^1f_m)\cong B$$ and we get 
$$\mathcal J^{\mathbb N}(B/k)\cong \mathcal J^{\mathbb N}(A/k)\otimes_AB.$$
We use the Jacobian criterion for smoothness. Considering the Jacobian matrix 
\[
\begin{pmatrix} \partial^1/\partial^1x_1(f_1) & \partial^1/\partial^1x_2(f_1) & \ldots &\partial^1/\partial^1x_n(f_1)\\
               \partial^1/\partial^1x_1(f_2) & \partial^1/\partial^1x_2(f_2) & \ldots  & \partial^1/\partial^1x_n(f_2)\\
               \vdots & \vdots & \ddots &\vdots  \\
               \partial^1/\partial^1x_1(f_m) & \partial^1/\partial^1x_2(f_m) & \ldots & \partial^1/\partial^1x_n(f_m)
               \end{pmatrix}
               \]
          
               We then have that the $n^{th}$ Fitting ideal of this matrix, generated by the $(n\times n)$-minors is equal to the unit ideal in $B$, because it is nonzero modulo each maximal ideal of $B$, and otherwise, if the fitting ideal where not the unit ideal, there would be a maximal ideal containing it, a contradiction.\\
               We want to consider the  $n^{th}$-fitting ideal  of the Jacobian matrix $\mathcal J(\underline{d^1f}/\underline{d^1x})$. To make sense of this, recall that by definition for the jet algebras for free polynomial algebras (see \cite{Guenther})[chapter 6.5, pp. 116-119], and simply by the fact that the universal derivation is a $k$-algebra homomorphism,
               \begin{gather*}f_i+d^1f_i =f_i(x_1+d^1x_1,..., x_n+ d^1x_n)=\\
               \sum_{I}\partial^{|I|}/\partial \underline{x}^I(f)\cdot \underline{d^1x}^I.
               \end{gather*}
               Observe that this sum is finite, since the $f_j$\, are polynomials and considering the $B$-algebra $B[d^1x_1,...,d^1x_n]/(d^1f_1,...,d^1f_m)$\, makes sense. So we can write $d^1f_i$\, as a polynomial in the $d^1x_i$\, with zero constant term and coefficients in $B$.\\
By the above formula, if $I=(0,...,1,...,0)$\,, we get that the first partial derivative of $d^1f_i$\, with respect to the free variable $d^1x_j$\, is just $\partial^1/\partial^1x_j(f_i)\in B$\,. For an arbitrary  multi-index $I$, this equality only holds up to a constant factor $c\in \mathbb N$\,. Thus, we can apply the Jacobian criterion for smoothness in order to conclude that
               $B[d^1x_1,...,d^1x_n]]/(d^1f_1,...,d^1f_m)$\, is etale over $B$, i.e., smooth of relative dimension zero.\\
We  consider the $n^{th}$ fitting ideal $\text{Fitt}^n$  of $\mathcal J(\underline{d^1f}/\underline{d^1x})$\,. The coefficients of the $n\times n$-minors lie actually in $B$. Let $\mathfrak{p}\in \Spec(B)$\, be given with $B_{\mathfrak{p}}/\mathfrak{p}\cdot B_{\mathfrak{p}} =K$\, being the residue field. We consider the reduced ring 
$$K[[d^1x_1,...,d^1x_n]]/(d^1f_1,...,d^1f_m).$$ The $n^{th}$ fitting ideal  of the Jacobian $\mathcal J(\underline{d^1f}/\underline{d^1x})$\, modulo $\mathfrak{p} $\, is then the unit ideal in $K$ which precisely means that the determinant of some $n\times n$-minor must be nonzero. But then, the determinant of this minor is a unit in $B_{\mathfrak{p}}$\, and there exists an open affine $\Spec C\subset \Spec B$\, such that this determinant is a unit in $A$. By the previous lemma, we conclude 
$$A[[d^1x_1,...,d^1x_n]/(d^1f_1,...,d^1f_m)\cong A.$$
 This holds for  each prime  ideal $\mathfrak{p}\in \Spec(B)$\,. Hence, there is a finite  Zariski open affine covering $\Spec B=\bigcup_{i=1}^n\Spec A_i$\, such that 
 $$A_i[[d^1x_1,...,d^1x_n]]/(d^1f_1,...,d^1f_m)\cong A_i$$
 and the claim follows. 
This shows, that the canonical homomorphism
$$\phi_M:\mathcal J^{\mathbb N}(M/k)\otimes_{A,p_1}B\longrightarrow \mathcal J^{\mathbb N}(M\otimes_AB/B),$$
which is simply the map 
$$ \widehat{A\otimes_kM}\otimes_{A,p_1}B\longrightarrow \widehat{B\otimes_kM}\longrightarrow \widehat{B\otimes_k(M\otimes_AB)}$$
is an isomorphism for $M=B$\,. Now, the proof is standard. $\phi_M$\, is a natural transformation of  right exact functors from $B-\Mod$\, to $\mathcal J^N(B/k)-\Mod$\,. The result follows for free $A$-modules $M$, since taking jet-modules commutes with taking direct sums and, choosing a free presentation for  general $M$, the result follows by the five-lemma. 
\end{proof}
\begin{corollary}\mylabel{cor:C19127} (invariance under  Zariski- localization) Let $k$ be a noetherian ring and  $k\longrightarrow A$\, be $k$-algebra of finite type. Let $S\subset A$\, be a multiplicatively closed subset. Then, there is a canonical isomorphism
$$\mathcal J^N(A_S/k)\cong \mathcal J^N(A/k)_S,$$
where $S=S\otimes 1$\, in $\mathcal J^N(A/k).$
\end{corollary}
\begin{corollary}\mylabel{cor:C19128} If $(A,\mathfrak{m},\kappa)$\, is a local ring that is a $k$-algebra essentially of finite type, then for each $N\in \mathbb N_0$\,, $\mathcal J^N(A/k)$\, is an $A$-algebra essentially of finite type.
\end{corollary}
 \subsection{Comparison with the $\mathcal C^{\infty}$-category}
 It is well known that  for $\mathcal C^{\infty}$-manifolds and vector bundles on them,  being a differential operator is a local property, which is in this category one way to define them. In this subsection we show that in the algebraic category, an analogous statement holds, if we use the etale topology on a smooth algebraic scheme $X$.
 \begin{proposition}\mylabel{prop:P5}Let $S$ be a noetherian  scheme  and $\pi: X\longrightarrow S$ be a smooth $S$- scheme  of finite type of dimension $n$\, over $S$ and $\mathcal E$\, be a locally free coherent $\mathcal O_X$-module. Let $D: \mathcal E\longrightarrow \mathcal E$\, be a homomorphism of  etale sheaves of $\pi^{-1}\mathcal O_S$-modules. Suppose, that for each scheme point $x\in X$, there is an etale neighbourhood $p_x: U_x\longrightarrow X$\, such that there is a trivialization $\phi_x: p_x^*\mathcal E\cong \mathcal O_{U_x}^{\oplus r}$\, plus an etale  surjective morphism $q_x: U_x\longrightarrow  V_x\subseteq \mathbb A^n_S$\,. Then $V_x$ is Zariski- open in $\mathbb A^n_S$\,. Let $\Gamma(U_x,D): \mathcal O_{U_x}^{\oplus r}\longrightarrow \mathcal O_{U_x}^{\oplus r}$\,  be the section over $U_x$ of $D$ with respect to the trivialization of  $\phi_x$\, of $\mathcal E$ around $x\in X$\,. We say that $D$ is a classical linear partial differential operator if there is a partial differential operator $D_x: \mathcal O_{V_x}^{\oplus r}\longrightarrow \mathcal O_{V_x}^{\oplus r} $\,, that pulls back under $q_x$\, to $\Gamma(U_x, D)$\,. Then, there is an $\mathcal O_X$-linear homomorphism $\widetilde{D}:\mathcal J^N(\mathcal E/S)\longrightarrow \mathcal E$\,  such that $D=\widetilde{D}\circ d^N_{\mathcal E/S}$\,, and conversely, every $D=\widetilde{D}\circ d^N_{\mathcal E/S}$\, is of this form.\\
Furthermore, every $\mathcal O_X$-linear homomorphism from $\Omega^{\leq N}(\mathcal E/S)\longrightarrow \mathcal E$\, (see \cite{Guenther}) corresponds to a classical differential operator $\mathcal E\longrightarrow \mathcal E$.
\end{proposition}
\begin{proof} Under these assumptions for each scheme point $x\in X$\, the classical operator $D_x: \mathcal O_{V_x}^{\oplus r}\longrightarrow \mathcal O_{V_x}^{\oplus r}$\, corresponds to a section 
$$\widetilde{D_x}'\in \Gamma(V_x, Hom_{V_x}(\mathcal J^{ N_x}(\oplus_{i=1}^r O_{V_x}/S), \oplus_{i=1}^r\mathcal O_{v_x})$$  for some $N_x\in \mathbb N$. Since the jet bundles are invariant under etale pull back (\prettyref{lem:L30111}) , for each $x\in X$\, we get a  the pulled back section 
$$\widetilde{D_x}\in \Gamma(U_x, Hom_{U_x}(\mathcal J^{N_x}(\mathcal E/S), \mathcal E)),$$
 using the trivialization of $\mathcal E$\, over $U_x$,  such that $\widetilde{D_x}$\, composed with $\Gamma(U_x,d^{ N_x}_{\mathcal E/S})$\, is $\Gamma(U_x,D)$\,. I claim that the $\widetilde{D_x}$\, glue to a global section of $DO^N_{X/S}(\mathcal E,\mathcal E)$\, over $X$ for some $N\in \mathbb N.$ First since $X$\, is quasicompact, we can find an etale  finite subcovering $\{U_{x_i}\longrightarrow  X,\quad i\in I\}$ with $I$\, a finite set. So we can take as our $N$\, the number $N=\max_{i\in I}N_{x_i}$. Now the global differential operators are a  subalgebra of the $\pi^{-1}(\mathcal O_S)$\, 
 linear endomorphism algebra of $\mathcal E$\,.  We know, that etale locally, the endomorphism $D$ is given by an $\mathcal O_X$-linear map $\widetilde{D_x}.$\,
  On etale overlaps $U_x\times_XU_y$, $D$ is certainly given by an element of $\Gamma(U_{xy}, DO^N(\mathcal E), \mathcal E)$\, 
 But since each element in $\Gamma(U_{xy}, DO^N(\mathcal E,\mathcal E))$\, determines uniquely an element in $Hom_{U_{xy}}(\mathcal J^{ N}(\mathcal E/S),\mathcal E)$\,
  the two elements obtained by restrictions from $U_x$\, and $U_y$\, to $U_{xy}=U_x\times_XU_y$\, must agree. Thus since $\mathcal J^{ N}(\mathcal E/S)$\, and $\mathcal E$\, are etale sheaves (since they are coherent on $X$),\, we get a global section in $DO^N(\mathcal E, \mathcal E)$\, over $X$.\\
 The converse of the statement follows from the fact, that if $\pi_x:U_x\longrightarrow V_y$\, is an etale morphism, then $$\pi_x^*DO^N(\mathcal O_{V_x}^{\oplus r}/S)\cong DO^N(\mathcal O_{U_x}^{\oplus r}/S),$$\, which is a simple consequence of the etale pull back property of the jet-modules (\prettyref{thm:T101}). Thus each differential operator on $U_x$\, is the pull back of a differential operator on $V_x$\,, so the local description of a globally defined differential operator on $\mathcal E/S$\, is always satisfied.\\
 The last statement follows from the local description of differential operators in the $\Omega$-formalizm on $\mathbb A^n_k$ (see \cite{Guenther}[chapter  6.4. Theorem 6.55,p.97, chapter 8, Corollary 8.11(2),p. 146], namely locally on $\mathbb A^n_k$\, they give classical partial linear differential operators, and the fact, that they form an etale sheaf .
 \end{proof}
 \begin{remark}\mylabel{rem:R22125} In the same situation, we can prove in the same way, that if $\mathcal E_1$\, and $\mathcal E_2$\, are locally free sheaves on $X$, then each linear partial differential operator between $\mathcal E_1$\, and $\mathcal E_2$\, has either a description via the jet bundle or the etale local description.
 \end{remark}
 \subsection{The global case}
 By the etale invariance property of the jet-module (and hence invariance under Zariski-localizations), if $q: X\longrightarrow S$\, is a morphism of finite type between noetherian schemes or noetherian algebraic spaces, if $\mathcal F$\, is a coherent sheaf on $X$, if we choose   affine Zariski-open covers of $X$ and $S$ , the locally defined jet-modules glue to a global jet-module $\mathcal J^N(\mathcal F/S)$\,. Under the assumptions made, this is a coherent sheaf on $X$. This follows from the fact, that the localization isomorphisms are canonical (follows from the universal representing properties of the jet-modules) and hence, the cocycle conditions are satisfied).\\
 If $X$ and $S$ are noetherian algebraic spaces, one defines the jet sheaf first in the case, where the morphism is representable, i.e. we can find an etale cover $\{\Spec A_i\longrightarrow S\}$\, such that $X\times_S\Spec A_i$\, is a scheme. Then, etale locally over $S$, the jet-modules are defined by the scheme case.\\
 If the morphism $q$ is not representable we can assume that $S=\Spec A$\, is a noetherian affine scheme. Then choose an etale cover $\{\Spec B_j\longrightarrow X\}$\, and the jet-modules $\mathcal J^N(M_j/A)$\,, where $\mathcal F\mid_{\Spec B_j}=\widetilde{M_j}$\, glue to a globally defined jet sheaf $\mathcal J^N(\mathcal F/S)$\,. All we need is the base change -and etale invariance property of the jet-modules. Also, the universal filtered derivations 
 $$d^N_{M_j/A_i}: M_j\longrightarrow \mathcal J^N(M_j/A_i), i\in I, j\in J$$ glue to a universal derivation 
 $$d^N_{\mathcal F/S}:  \mathcal F\longrightarrow \mathcal J^N(\mathcal F/S).$$
  Furthermore, for a fixed quasi coherent sheaf $\mathcal F$\,, the universal representing property of the pair $d^N_{\mathcal F/S},\mathcal J^N(\mathcal F/S)$\, for the moduli problem, sending a $\mathcal J^N(X/S)$-module $\mathcal Q$\, to the set of all  filtered derivations $t: \mathcal F\longrightarrow \mathcal Q$\, is satisfied, because the required homomorphism of $\mathcal J^N(X/S)$-modules $\phi: \mathcal J^N(\mathcal F/S)\longrightarrow \mathcal Q$\, can be constructed etale-or Zariski-locally, and by the universal property in the affine case, these locally construced $\phi_i$\, glue to a global $\phi$\,. If $\phi_1,\phi_2$\, are two homomorphisms of $\mathcal J^N(X/S)$-modules with $t=\phi_i\circ d^N_{\mathcal F/S}$\,, then they locally agree, hence by the sheaf property they agree globally.  We have proved the following
 \begin{theorem}\mylabel{thm:T20127} Let $q: X\longrightarrow S$\, be a morphism of finite type of noetherian schemes, or, more generally of neotherian algebraic spaces and let $\mathcal F$\, be a quasi coherent sheaf on $X$. Then,  for each $N\in \mathbb N_0\cup \{\mathbb N\},$\, there is a quasi coherent $\mathcal O_X$-module $\mathcal J^N(\mathcal F/S)$\, plus a filtered derivation
 $$d^N_{\mathcal F/S}: \mathcal F\longrightarrow \mathcal J^N(\mathcal F/S)$$
 with respect to the diagonal ideal sheaf $\mathcal I_{X/S}$\, that represents the functor, sending a 
 $\mathcal J^N(X/S)$-module $\mathcal Q$\, to the set of all filtered derivations $t: \mathcal F\longrightarrow \mathcal Q$\, with respect to the diagonal ideal sheaf $\mathcal I_{X/S}$\,.
 \end{theorem}
 \begin{proof}
 \end{proof}
 We thus make the following  (basically standard) definition.
 \begin{definition}\mylabel{def:D46} Let $X\longrightarrow S$\, be an arbitrary morphism of finite type of  noetherian schemes, or more generally of  noetherian algebraic spaces and  $\mathcal F_i, i=1,2$\, be  quasi coherent sheaves on $X$. Then, a differential operator of order $\leq N$\, is an $\mathcal O_S$-linear map $D: \mathcal F_1\longrightarrow \mathcal F_2$\, that can be factored as
  $$\mathcal F_1\stackrel{d^N_{\mathcal F/S}}\longrightarrow \mathcal J^N(\mathcal F_1/S) \stackrel{\widetilde{D}}\longrightarrow \mathcal F_2,$$
  where the homomorphism $\widetilde{D}$\, is $\mathcal O_X$-linear, where $\mathcal J^N(X/S)$\, is regarded with respect to the $\mathcal O_X$-module structure coming from the first tensor factor.
  A differential operator of order $N$\, is a differential operator that is of order $\leq N$\, but not of order $\leq N-1$\,.
 \end{definition}
 Thus, in this situation, there is a 1-1 correspondence between differential operators $\mathcal F_1\longrightarrow \mathcal F_2$\, relative to $S$ and $\mathcal O_X$-linear maps $\mathcal J^N(\mathcal F_1/S)\longrightarrow \mathcal F_2$\,.
\begin{proposition}\mylabel{prop:P2} (arbitrary push-forwards) Let $X\stackrel{f}\longrightarrow Y\stackrel{p}\longrightarrow S$\, be morphisms of schemes and $\mathcal F_i, i=1,2$\, be  quasi coherent sheaves on $X$. Let $D: \mathcal F_1\longrightarrow \mathcal F_2$\, be a differential operator relative to $S$. Then $f_*F_1\stackrel{f_*D}\longrightarrow f_*F_2$\, is a differential operator between the quasi coherent sheaves $f_*F_i$\, relative to $S$, where $f_*D$\, is taken in the category of sheaves of $(\pi\circ f)^{-1}\mathcal O_S$-modules on $X$.
\end{proposition}
\begin{proof} Let $D$ be given by 
$$\widetilde{D}\circ d^N_{\mathcal F_1/S}: \mathcal F_1\longrightarrow \mathcal J^N(\mathcal F_1/S)\longrightarrow \mathcal F_2,$$ where the first map is $(\pi\circ f)^{-1}\mathcal O_S$-linear and $\widetilde{D}$\, is $\mathcal O_X$-linear. Then $f_*d^N_{\mathcal F_1/S}$\, is an $\pi^{-1}\mathcal O_S$-linear map from $f_*\mathcal F_1$\, to $f_*\mathcal J^N(\mathcal F_1/S)$\, and $f_*\widetilde{D}$\, is $f_*\mathcal O_X$, and thus $\mathcal O_Y$\,
 linear via the structure homomorphism $\mathcal O_Y\longrightarrow f_*\mathcal O_X$\,.\\
  The morphism $f$ induces a morphism 
  $$J^N(f/S): J^N(X/S)\longrightarrow J^N(Y/S),$$
   where $J^N(X/S):=\Spec_X\mathcal J^N(X/S)$\, with  projection $p_{X}:J^N(X/S)\longrightarrow X$\,,  such that
   $$ p_{Y}\circ J^N(f/S)=f\circ p_X.$$ 
   and thus we have a homomorphism of sheaves $\mathcal J^N(Y/S)\longrightarrow f_*\mathcal J^N(X/S)$\,.\\
    Hence we have  that $f_*\mathcal J^N(\mathcal F/S)$\, is an $f_*\mathcal J^N(X/S)$-module and thus an $\mathcal J^N(Y/S)$-module.\\
 By \prettyref{lem:L45}, there is a unique homomorphism 
 $$\phi: \mathcal J^N(f_*\mathcal F_1/S)\longrightarrow f_*\mathcal J^N(\mathcal F_1/S)$$ such that $f_*d^N_{\mathcal F_1/S}=\phi\circ d^N_{f_*\mathcal F_1/S}$\,.   The $\pi^{-1}\mathcal O_S$-linear map $f_*D$\, can be written as
 $$f_*D: f_*\mathcal F_1\stackrel{d^N_{f_*\mathcal F_1/S}}\longrightarrow \mathcal J^N(f_*\mathcal F_1/S)\stackrel{(f_*\widetilde{D}\circ \phi)}\longrightarrow f_*\mathcal F_2$$
 and  is a partial linear differential operator on $Y$ over $S$.
\end{proof} 

\begin{remark}\mylabel{rem:R30111} If $q: X\longrightarrow S$\, is a morphism of noetherian schemes, $D: \mathcal E_1\longrightarrow \mathcal E_2$\, is a differential operator relative to $S$ and $\mathcal F$\, is a quasi coherent $\mathcal O_S$-module, it follows from  the global version of \prettyref{lem:L295} that $D\otimes_{q^{-1}\mathcal O_S}\text{Id}_{\mathcal F}$\, is a differential operator on $X$ relative to $S$.
\end{remark}
\begin{proposition}\mylabel{prop:P3} (etale pull back) Let $X\stackrel{f}\longrightarrow Y\stackrel{\pi}\longrightarrow S$\, be morphisms of schemes where $f$ is etale. If $D:\mathcal F\longrightarrow \mathcal F$\, is a differential operator on the quasi coherent $\mathcal O_Y$-module $\mathcal F$\,, then $f^*D: f^*\mathcal F\longrightarrow f^*\mathcal F$\, is a differential operator on the quasi coherent $\mathcal O_X$-module $f^*\mathcal F$\,.
\end{proposition}
\begin{proof} This follows from the etale invariance property of the jet modules, i.e. $f^*\mathcal J_Y^N(\mathcal F/S)\cong \mathcal J^N_X(f^*\mathcal F/S)$\, (\prettyref{lem:L30111}.) Then, if $D$ is given as $D=\widetilde{D}\circ d^N_{\mathcal F/S}$\,,  the homomorphism of sheaves of $(\pi^{-1}\circ f)(\mathcal O_S)$-modules $f^*D$, is given by 
$$f^*D: f^*\mathcal F\stackrel{f^*d^N_{\mathcal F/S}=d^N_{f^*\mathcal F/S}}\longrightarrow f^*\mathcal J^N(\mathcal F/S)\cong \mathcal J^N(f^*\mathcal F/S)\stackrel{f^*\widetilde{D}}\longrightarrow f^*\mathcal F.$$
\end{proof}

To avoid confusion, for each $N\in \mathbb N_0,$ the $\mathcal O_S$-module  $\mathcal J^N(X/S)$\, can be regarded as the structure sheaf of the higher tangent bundle $J^N(X/S)$\,. This is an $\mathcal O_X$-bi-module with respect to the two tensor factors, sloppily written as $\mathcal J^N(X/S)=\mathcal O_X\otimes_{\mathcal O_S}\mathcal O_X/\mathcal I_{X/S}^{N+1}$\,. Denote by $p_{1,X},p_{2,X}$\, the two projections $J^N(X/S)\longrightarrow X$\,, where we defined in the introduction $p_{1,X}=p_X$\,. If the scheme $X$ under consideration is clear from the context, we drop the subscript $(-)_X$\,.\\
 To be more precise, there are two $\mathcal O_S$-linear homomorphisms 
 $$\mathcal O_X\stackrel{p_1^{\sharp},p_2^{\sharp}}\longrightarrow \mathcal J^{\mathbb N}(X/S).$$ If locally $\Spec A\subset S$ and $\Spec B\subset X$\, are open affine subsets, $\Spec B$ mapping to $\Spec A$\,, then 
 $$\Gamma(\Spec A, \mathcal J^{ N}(X/S))=B\otimes_AB/I_{B/A}^{N+1}$$ The two maps $p_1^{\sharp}, p_2^{\sharp}$\, correspond to the natural maps 
$$B\longrightarrow B\otimes_AB\longrightarrow B\otimes_AB/I_{B/A}^{N+1},\quad b\mapsto b\otimes 1, 1\otimes b.$$ Both homomorphisms give $\mathcal J^{ N}(X/S)$\, the structure of a quasi coherent $\mathcal O_X$-algebra We  have defined in the section Notation and Conventions  
$$J^{N}(X/S)=\Spec_{X}p_{1,,X*}\mathcal J^{ N}(X/S)$$ with natural projection
$p_{1,X}=p_X: J^{ N}(X/S)\longrightarrow X$\, which is a morphism of schemes over $S$. There is a second morphism over $S$, $p_{2,X}: J^{ N}(X/S)\longrightarrow X$\, whose structure homomorphism 
$$\mathcal O_X\longrightarrow p_{2,X,*}\mathcal J^N(X/S)$$  corresponds to the universal filtered derivation. This holds for all $N\in \mathbb N_0\cup\{\mathbb N\}.$\\
This we want to make clear by the following
\begin{lemma}\mylabel{lem:L46} Let $q:X\longrightarrow S$\, be a morphism of schemes and $\mathcal Q$\, be a $\mathcal J^N(X/S)$\,-module for some $N\in \mathbb N_0$. Then $p_{1,*}\mathcal Q\cong p_{2,*}\mathcal Q$\, as $q^{-1}\mathcal O_S$-modules.\\
In particular, 
$$p_{X,1,*}\mathcal O_{J^N(X/S)}=p_{X,2,*}\mathcal O_{J^N(X/S)}=\mathcal J^N(X/S).$$
\end{lemma}
\begin{proof} The morphisms $p_1$\, and $p_2$\, are affine and finite and on the underlying scheme points a topological isomorphism $\mid J^N(X/S)\mid \cong \mid X\mid$\,.  Let $\Spec B\subset X$\, and $\Spec A\subset S$\, be open affine subschemes with $\Spec B$\, mapping to $\Spec A$\,. Then 
$$p_1^{-1}(\Spec B)=p_2^{-1}(\Spec B)=\Spec(B\otimes_AB/I_{B/A}^{N+1}).$$ Then by definition of push forward of a sheaf, the claim follows.
\end{proof}
\begin{remark}\mylabel{rem:R47} If $\mathcal Q$\, is an $\mathcal J^N(X/S)$-module, the sheaf $p_{1,*}\mathcal Q=p_{2,*}\mathcal Q$\, simply regarded as a sheaf of $q^{-1}\mathcal O_S$-modules  on $X$, possesses two $\mathcal O_X$-module structures. Restricting to an open affine $\Spec B\subset X$\,, if $\mathcal Q$\, corresponds to the $B\otimes_AB/I_{B/A}^{N+1}$-module $M$ , this is simply the
 $A$-module $M$, and the two $\mathcal O_X$-module structures on $M$ correspond to the two $B$-algebra structures on $\mathcal J^N(B/A)$\,.
 \end{remark}
 \begin{lemma}\mylabel{lem:L10} With notation as above, suppose that $\supp(\mathcal F)=Y\subsetneq X$\, with $\mathcal I_Y=\text{ann}(\mathcal F)$\,. Then, if there is a differential operator $\mathcal F\stackrel{D_Y}\longrightarrow \mathcal F$\, on $Y$, i.e., $\mathcal F$\, regarded as a sheaf on $Y$, then too on $\mathcal F$\, regarded as a sheaf on $X$.
\end{lemma}
\begin{proof} By assumption, there is some $N\in \mathbb N$\, plus an $\mathcal O_Y$-linear map $\widetilde{D_Y}:\mathcal J^N_Y(\mathcal F/k)\longrightarrow \mathcal F$\,. By   looking at the local description of the jet-modules, there is always an $\mathcal O_X$-linear surjection $p_{XY}:\mathcal J^N_X(\mathcal F/k)\longrightarrow \mathcal J^N_Y(\mathcal F/k)$\, (which is locally of the form 
$$(\overline{A\otimes_k M}\twoheadrightarrow \overline{(A/I_Y)\otimes_kM)}$$
which is $(A,p_1)$-linear.  Composing with $p_{XY}$\,, we get $\widetilde{D_X}=\widetilde{D_Y}\circ p_{XY}: \mathcal J_X^N(\mathcal F/k)\longrightarrow \mathcal F$\,, that, composed with $d^N_{X, \mathcal F/k}$, gives the differential operator over $X$, $D_X: \mathcal F\longrightarrow \mathcal F$\,.
\end{proof}
In order to study the behavior of a differential operator with respect to the natural torsion filtration on a coherent sheaf, we prove the following 
\begin{lemma}\mylabel{lem:L12} Let $\pi: X\longrightarrow S$\, be a   morphism of algebraic schemes and $\mathcal F$\, be a quasi coherent sheaf on $X$ and 
$\mathcal F'\subset \mathcal F$\, be a coherent subsheaf. For each $N\in \mathbb N,$\, let $\mathcal J^N(\mathcal F'/S)'$\, be the subsheaf of $\mathcal J^N(\mathcal F/S)$\, which is the image under the natural homomorphism $\mathcal J^N(F'/S)\stackrel{\mathcal J^N(i/S)}\longrightarrow \mathcal J^N(F/S)$, where $i: \mathcal F'\hookrightarrow \mathcal F$ is the inclusion. We have  on $X$ 
$$\ann(\mathcal J^N(\mathcal F'/S)')\supseteq \text{ann}(\mathcal F')^{N+1}\otimes_{\mathcal O_S}\mathcal O_X,$$ where we regard $\mathcal J^N(\mathcal F'/S)$\, as a coherent sheaf on $J^N(X/S)$. Thus, if $\dim(\mathcal F')\leq d$, then also 
$\dim(\mathcal J^N(\mathcal F'/S)')\,\leq d$.
\end{lemma}
\begin{proof} The question is local, so let $A\longrightarrow B$\, be a  homomorphism of finitely generated $k$-algebras and let $M$ be an $A$- module. I claim that $\text{ann}(M)^{N+1}\otimes_AB\subseteq \text{ann}(\mathcal J^N(M/A))$\,. We have 
$$\mathcal J^N(M/A)=B\otimes_AM/I_{B/A}^{N+1}\cdot (B\otimes_AM).$$ Let $\mathfrak{a}=\text{ann}(M)$\, and $a\in \mathfrak{a}$\,. By definition, we know, that $\mathcal J^N(M/A)$\, is annihilated by $B\otimes \mathfrak{a}$\,. We have 
$a\otimes 1-1\otimes a\in I_{B/A}$\, and  for all $m\in M$,

$$0=(a\otimes 1-1\otimes a)^{N+1}\cdot 1\otimes m=(a^{N+1}\otimes 1+ (1\otimes a)\cdot \omega)\cdot 1\otimes m$$
 and it follows $(a\otimes 1)^{N+1}\cdot 1\otimes m =0\,\,\forall m\in M \,\,\text{and}\, a\in \mathfrak{a}$\,.
 Thus 
 $$\text{ann}(M)^{N+1}\otimes_AB\subseteq \ann(\mathcal J^N(M/A)).$$ If now $M'\subset M$\, corresponds over $\Spec B$\, to $\mathcal F'\subset \mathcal F$\, and $\mathfrak{a}'=\text{ann}(M)'$§\, then $\ann(\mathcal J^N(M'/A)\supseteq \mathfrak{a'}^{N+1}\otimes_AB$\,  and so the image $\mathcal J^N(F'/S)'\subseteq \mathcal J^N(\mathcal F/S)$\, is also (locally over $X$\, annihilated by $\mathfrak{a'}^{N+1}\otimes_AB$\,. 
 Now the statement about the dimension follows from the fact that  if $\dim(\mathcal O_X/\ann(\mathcal F'))\leq d$\,, then also 
 \begin{gather*}d \geq  \dim(\mathcal F')=\dim(\mathcal O_X/\ann(\mathcal F')^{N+1})\geq \\
 \dim(\mathcal O_{J^N(X/S)}/\ann(\mathcal J^N(\mathcal F'/S))')=\dim(\mathcal J^N(F'/S)').
 \end{gather*}
\end{proof}
We have the following important
\begin{corollary}\mylabel{cor:C11} Let $\pi:X\longrightarrow S$\, be a morphism of algebraic schemes and $\mathcal E$\, be a coherent $\mathcal O_X$-module. Let $T^i\mathcal E, i=0,...,\dim(\mathcal E)$\, be the torsion filtration of $\mathcal E$ and for some $N\in \mathbb N$\, and $D: \mathcal E\longrightarrow \mathcal E$\, be a differential operator relative to $S$ of order $\leq N$\,. Then, $D$ respects the torsion filtration of $\mathcal E$\,, i.e. $D(T^i(\mathcal E))\subseteq T^i\mathcal E.$\,
\end{corollary}
\begin{proof} The question is local so let be as above $A\longrightarrow B$\, be a homomorphism of rings and $M$ be a $B$ module and $D: M\longrightarrow M$\, be a differential operator of order $\leq N$\,. Let $M'\subset M$\, be a submodule of $M$ of dimension $\leq d$\,. Let $I=\text{ann}(M')$\,. Then by the previous proposition $I^{N+1}\otimes B\subseteq \text{ann}(\mathcal J^N(M'/A)$\,. The differential operator $D$\, restricted to $M'$\, factors over $\mathcal J^N(M'/A)'\subset \mathcal J^N(M/A)$\,. Let $\widetilde{D}: \mathcal J^N(M/A)\longrightarrow M$\, be the $B$-linear map corresponding to $D$. Then the image of $\mathcal J^N(M'/A)$\, in $\mathcal J^N(M/A)$\, is likewise annihilated by $I^{N+1}\otimes B$\, and so is $\widetilde{D}(\mathcal J^N(M'/A))\subset M$\,. Thus, the image of $M'$\, under $D$ in $M$ is contained in a submodule annihilated by $I^{N+1}$\,. Since $\dim(M')\leq d$ $\dim(B\cdot D(M'))\leq d.$ $T^d(M)$\, is the maximal submodule of $M$ of dimension $\leq d$\, and we have proved that $\dim(B\cdot D(T^d(M))\leq d$\, which implies the claim.
\end{proof} 
 Let $\pi: X\longrightarrow S$\, be a smooth morphism  of finite type of noetherian schemes and $\mathcal E$\, be a coherent sheaf on $X$. Let $\mathcal J^{(N)}(\mathcal E/S)=I_{X/S}\cdot \mathcal J^N(\mathcal E/S)$\, so in particular $\mathcal J^{(N)}(X/S)=I_{X/S}/I_{X/S}^{N+1}$\,. There is a short exact sequence
 $$ (*)\,\,0\longrightarrow \mathcal J^{(N)}(\mathcal E/k)\longrightarrow \mathcal J^N(\mathcal E/k)\longrightarrow \mathcal E\longrightarrow 0.$$ For a smooth morphism, it is well known that $I_{X/S}^N/I_{X/S}^{N+1}\cong \Omega^{(1)}(X/S)^{\otimes^sN}$\,.
 \begin{lemma}\mylabel{lem:L33} With notation as above, the  homomorphism
 $$I_{X/S}^N/I_{X/S}^{N+1}\otimes_{\mathcal O_X}\mathcal J^N(\mathcal E/S)\longrightarrow \mathcal J^N(\mathcal E/S),$$
  coming from the $\mathcal J^N(X/S )$-module structure of $\mathcal J^N(\mathcal E/S)$\, descends to a $\mathcal O_X$-linear map,
  $$I_{X/S}^N/I_{X/S}^{N+1}\otimes_{\mathcal O_X}\mathcal E\longrightarrow \mathcal J^N(\mathcal E/S).$$
  \end{lemma}
  \begin{proof} This follows from the exact sequence (*) and the fact that $I_{X/S}^N/I_{X/S}^{N+1}\otimes I_{X/S}\cdot \mathcal J^N(X/S)\longrightarrow \mathcal J^N(\mathcal E/S)$\, is the zero map.
  \end{proof}
  We now state here the following basic fact about jet-modules in the global case.
  \begin{proposition}\mylabel{prop: P4} Let $\pi: X\longrightarrow S$\, be a morphism of finite type of noetherian schemes. 
 \begin{enumerate}[1]
\item  for each $N\in \mathbb N_0\cup\{\mathbb N\}$\,, let $\mathcal J^{N}(-/S)$\, be the functor from quasi coherent $\mathcal O_X$-modules  to quasi coherent $\mathcal J^{ N}(X/S)$-modules, sending $\mathcal F$ to $\mathcal J^{ N}(\mathcal F/S)$\,. Then, this functor is right  exact and there is a canonical natural isomorphism $\mathcal J^{ N}(-/S)\stackrel{\cong}\longrightarrow p_2^*(-).$
\item If $\pi:X\longrightarrow S$\, is flat, then $\mathcal J^{N}(-/S)$\, is an exact functor.
\item If $\pi: X\longrightarrow S$\, is a smooth morphism of noetherian schemes, then for each $N\in \mathbb N_0$\,, the functor $\mathcal J^N(-/S)$, sending quasi coherent $\mathcal O_X$-modules to quasi coherent $\mathcal J^N(X/S)$-modules, is exact and equal to $(p^N_2)^*$\,.
\end{enumerate}
\end{proposition}
\begin{proof}
\begin{enumerate}[1]
\item This follows from the local definition of the jet-modules. Here, of coarse, $p_2: J^N(X/S)\longrightarrow X$\, is the finite affine morphism which corresponds to the second $\mathcal O_X$-module structure.
\item This immediately follows from (1).
\item  This follows from the fact, that for $X/S$\, smooth, the $N^{th}$-jet algebra $\mathcal J^N(X/S)$\, is a projective, hence flat $\mathcal O_X$-module (see \prettyref{prop:P20124}). So the assertion follows from (2).
\end{enumerate}
\end{proof}
\begin{proposition}\mylabel{prop:P6}(Exact sequence I) Let $\pi:X\longrightarrow S$\, be a morphism of finite type between noetherian schemes and $Y\subset X$\, be a closed subscheme and $\mathcal F$\, be a quasi coherent $\mathcal O_Y$-module.
Let $\mathcal I_Y$\, be the defining ideal sheaf of $Y$. Then  for all $N\in \mathbb N\cup \{\infty\}$\, there is an exact sequence
$$0\longrightarrow  (\mathcal I_Y)\cdot \mathcal J^N_X(\mathcal F/S)\longrightarrow \mathcal J_X^N(\mathcal F/S)\longrightarrow \mathcal J^N_Y(\mathcal F/S)\longrightarrow 0.$$
Here multiplication with $\mathcal I_Y$\, is via the first $\mathcal O_X$-module structure.
\end{proposition}
\begin{proof} This is a local question , so let $A\longrightarrow B$\, be a homomorphism of rings, $I\subset B$\, be an ideal and $M$ be a $B/I$-module. We have 
\begin{gather*}\mathcal J_{B/A}^N(M/A)=B\otimes_AM/I_{B/A}^{N+1}\cdot (B\otimes_AM)\quad \text{and}\\ 
\mathcal J_{(B/I)/A}(M/A)=(B/I)\otimes_AM/I_{(B/I)/A}^{N+1}\cdot (B/I\otimes_AM).
\end{gather*} $\overline{B\otimes_AM}$\, is already a $B\otimes_A(B/I)$\, module, so tensoring with $B/I$\,via the first $B$-module structure, we obviously get $\overline{B/I\otimes_AM}$\, which is $\mathcal J_{(B/I)/A}(M/A)$\,.
\end{proof}
\begin{proposition}\mylabel{prop:P66} (Exact sequence II) Let as above $\pi: X\longrightarrow S$\,be a morphism of finite type between noetherian schemes and $Y\hookrightarrow X$\, be a closed subscheme with defining ideal sheaf $\mathcal I_Y$\,. Let $\mathcal F$\, be a quasi coherent $\mathcal O_Y$-module. Then  for all $N\in \mathbb N_0\cup \{\mathbb N\}$\, there is an exact sequence 
$$0\longrightarrow \mathcal I_Y\cdot \mathcal J^{N}_X(\mathcal F/S)\mid_Y\longrightarrow \mathcal J_X^N(\mathcal F/S)\mid_Y\longrightarrow \mathcal J^N_Y(\mathcal F\mid_Y/S)\longrightarrow 0,$$
 where multiplication with $\mathcal I_Y$\, is via the second $\mathcal O_X$-module structure on the jet bundle.
\end{proposition} 
\begin{proof} The question is again local, so let $A\longrightarrow B$\, be a homomorphism of rings and $M$ be a $B$-module and $I\subset B$\, be the ideal corresponding to $Y\subset X$\,. The module $\mathcal J^N_{Y/S}(F|_Y/S)$\, then corresponds to the module $\overline{(B/I)\otimes_A(M/IM)}$\, (modulo the ideal $I_{(B/I)/A}^{N+1}$) and the jet bundle $\mathcal J^N_X(\mathcal F/S)\mid_Y$\, corresponds to the module $\overline{B/I\otimes_AM}$\, so tensoring with $B/I$\, via the second $B$-module structure we get $\overline{B/I\otimes M/IM}$\, which was the claim.
\end{proof} 
\begin{proposition}\mylabel{prop:P7}(First cotangential sequence)\\
 Let $E\stackrel{i}\hookrightarrow X\stackrel{\pi}\longrightarrow S$\, be morphisms of finite type of noetherian   schemes such that $i$ is a closed immersion. Let $\mathcal F$\, be a quasi coherent $\mathcal O_X$-module. Then there is and exact sequence
$$(*_1)\quad \mathcal I_E\cdot \mathcal F/\mathcal I_E^{2}\cdot \mathcal F\longrightarrow \mathcal J_X^{(1)}(\mathcal F/S)\mid_E\longrightarrow \mathcal J_E^{(1)}(\mathcal F\mid_E/S)\longrightarrow 0.$$
If  $\pi$\, and $\pi\circ i$\, are smooth and $\mathcal F$\, is locally free,  the sequence $(*_1)$\, is exact on the left.\\
In this case, the first module is isomorphic to $\mathcal I_E/\mathcal I_E^{2}\otimes_{\mathcal O_X}\mathcal F$.\\
These exact sequences are functorial in $\mathcal F\in \text{QCoh}(X)$\,.
\end{proposition} 
\begin{proof}
 In  view of  the exact sequence II, it is enough to  construct  a functorial homomorphism
$$ \phi: \mathcal I_E\cdot \mathcal F/\mathcal I_E^2\cdot \mathcal F \longrightarrow \mathcal I_E\cdot \mathcal J^{(1)}(\mathcal F/S)\mid_E.$$
The question is local, so let $k\longrightarrow A$\, be a homomorphism of commutative rings, $I\subset A$\, be an ideal, and $M$\, be an $A$-module. We have to construct an $A$-linear homomorphism
$$\psi_M: IM/I^2M\longrightarrow (A/I)\otimes_k M/(I_{A/k}^2\cdot ((A/I)\otimes_kM)),$$
which is functorial in $M$, i.e., is a natural transformation from $A-\Mod$\, to $A-\Mod$\,.
Let $f\in I, m\in M$\, be given. We let $\psi_M(fm):=\overline{1\otimes fm}$. This is obviously additive. Secondly, since $0=(f\otimes 1-1\otimes f)^2= f^2\otimes 1 +2f\otimes f+1\otimes f^2=1\otimes f^2$\,, it follows that $I^2M$ maps to zero.  We have to show that this is $A$-linear. So for $a\in A$\, we have to show that 
$$\psi_M(afm)=\overline{1\otimes afm}=\overline{a\otimes fm}\quad \forall a\in A, f\in I, m\in M.$$
Now,$\mathcal J^1(M/k)$\, is an $\mathcal J^1(A/k)=A\oplus \Omega^{(1)}(A/k)$-module. There is the standard $A$-linear homomorphism $I/I^2\longrightarrow \Omega^{(1)}(A/k)\otimes_AA/I$\, which is in our notation the homomorphism
$\psi_A$\, with $\psi_A(f)=\overline{1\otimes f-f\otimes 1}=\overline{1\otimes f}$\,. Thus we have in 
$\mathcal J^1(A/k)\otimes_AA/I$\, the identity $\overline{1\otimes af}=\overline{a\otimes f}$\,. Then from the $\mathcal J^1(A/k)$-module structure of $\mathcal J^1(M/k)$\, it follows that 
$\overline{1\otimes afm}=\overline{a\otimes fm}$\, which is the $A$-linearity of the map $\psi_M$\,.\\
Thus we have in any case an exact sequence 
$$\mathcal I_E\cdot \mathcal F/\mathcal I_E^2\cdot \mathcal F\longrightarrow \mathcal J^{(1)}_{X/S}(\mathcal F/S)\mid_E\longrightarrow \mathcal J^{(1)}_{E/S}(\mathcal F\mid_E/S)\longrightarrow 0.$$

 So let now $\pi$\, and $\pi\circ i$\, be smooth.
  From the standard cotangential sequence, the claim is true for $\mathcal F\cong \mathcal O_X$\,. Since the assertion is local,   the claim is true for locally free $\mathcal O_X$\,-modules in which case 
  $$\mathcal I_E/\mathcal I_E^2\otimes_{\mathcal O_X}\mathcal F\cong \mathcal I_E\cdot \mathcal F/\mathcal I_E^2\cdot \mathcal F.$$ Both sides are right exact functors from $\mathcal O_X$-modules to $\mathcal O_E$-modules, so the claim follows by taking locally a free presentation and the five lemma. 
\end{proof}
\begin{proposition}\mylabel{prop:P8} (Second cotangential  sequence))\\
Let $X\stackrel{f}\longrightarrow Y\stackrel{g}\longrightarrow Z$\, be morphisms of  finite type of noetherian schemes and $\mathcal E$\, be a quasi coherent $\mathcal O_Y$-module. Then  there is an exact sequence
$$ (*_2)\quad f^*\mathcal J^{(1)}(\mathcal E/Z)\stackrel{\alpha_{\mathcal E}}\longrightarrow \mathcal J^{(1)}(f^*\mathcal E/Z)\stackrel{\beta_{\mathcal E}}\longrightarrow f^*\mathcal E\otimes_{\mathcal O_X}\mathcal J^{(1)}(X/Y)\longrightarrow 0.$$
If $f$ and $g$ are smooth and $\mathcal E$\, is locally free, then the left hand map is injective.\\
These sequences are functorial in $\mathcal E\in \text{QCoh}(Y).$
\end{proposition}
\begin{proof}
\begin{enumerate}[1]
\item   Construction of  the exact sequence\\
The question is local so, let $X=\Spec A$, $Y=\Spec B$\, $Z=\Spec C$\, and let $\mathcal E$\, correspond to the $B$-module $M$. The generalized second cotangential sequence is then a subsequence (the trivial direct summand $f^*\mathcal E$\, deleted) of the following sequence
\begin{gather*}C\otimes_{B,p_1}\overline{B\otimes_AM}=\overline{C\otimes_AM}\longrightarrow \overline{C\otimes_A(C\otimes_BM)}\\
\longrightarrow \overline{C\otimes_B(C\otimes_BM)} =(C\otimes_BM)\otimes_B\overline{B\otimes_BC},\longrightarrow 0
\end{gather*}
where the last identity comes from the fact that 
$$\mathcal J^N(B/A)\otimes_AM\cong \mathcal J^N(B\otimes_AM/A)$$ which is a special case of base change for jet modules.
Since the homomorphisms of the sequence are canonical, these glue to the sequence $(*_2)$\,. Observe that in $\mathcal J^1(M/k)$\,, $\mathcal J^{(1)}(M/k)$\, sits as a direct summand. Observe furthermore, that each three terms in the sequence $(*_2)$ are functors from $\mathcal O_Y-\Mod$\, to $\mathcal O_X-\Mod$\, that are all right exact and $\alpha_{\mathcal E}$\, and $\beta_{\mathcal E}$\, are natural transformations of functors from $\mathcal O_Y-\Mod$\, to $\mathcal O_X-\Mod$\,.
\item  Exactness\\
 If $\mathcal E\cong \mathcal O_Y$\, this is the standard  cotangential exact sequence. For arbitrary  locally free $\mathcal E$\, the question is local , so we may assume that $X =\Spec A,Y=\Spec B,Z=\Spec C $ are affine such that $\mathcal E$\, is free on $\Spec B$\,. Then from the standard cotangential sequence for the situation
$$\Spec A\longrightarrow \Spec B\longrightarrow \Spec C$$ and taking direct sums, we know that the sequence is locally exact, hence globally  exact.\\
For arbitrary coherent $\mathcal E$\,, we only need to show that the sequence is locally exact, since this is a statement about sheaves. If $\Gamma(\Spec B,\mathcal E)=M$\,  choose a presentation 
$$B^{\oplus I}\longrightarrow B^{\oplus J}\longrightarrow M\longrightarrow 0.$$ Then the second cotangential sequence for $M$ is the cokernel of the homomorphism of the second cotangential exact sequence for $A^{\oplus I}$\, to $A^{\oplus J}$\,. (all three terms are  right exact functors from $\mathcal O_Y-\text{Mod}$ to $\mathcal O_X-\text{Mod}$\,. Since the cokernel functor is right exact, the claim follows. \\
If $f$ and $g$ are smooth and $\mathcal E$\, is locally free, this reduces locally to $\mathcal E\cong \mathcal O_Y$\, where this is the standard second cotangential exact sequence.
\end{enumerate}
\end{proof}
\begin{lemma}\mylabel{lem:L360} Let $q: X\longrightarrow S$\, be a morphism of finite type of noetherian schemes where $X$ possesses an ample invertible sheaf and $D: \mathcal F_1 \longrightarrow \mathcal F_2$\, be a differential operator of order $\leq N$\, relative to $S$ for some $N\in \mathbb N_0$\, between quasi coherent sheaves $\mathcal F_1$\, and $\mathcal F_2$\,,  Then, there is a filtration $F^{\bullet}\mathcal F_1$\, and a filtration $\mathcal F^{\bullet}(\mathcal F_2)$\, such that $F^j(\mathcal F_i), i=1,2$\, is coherent and $D(\mathcal F^j_1)\subseteq \mathcal F^j_2$\,, where $j\in J$\, and $J$ is the directed index set of the filtration.
\end{lemma}
\begin{proof}  Under the assumptions made, each quasi-coherent sheaf $\mathcal E$\, is the direct limit of its coherent subsheaves $\mathcal E^i\subset \mathcal E $\,. Write $D$ as 
$$\mathcal F_1\stackrel{d^N_{\mathcal F_1}}\longrightarrow \mathcal J^N(\mathcal F_1/S)\stackrel{\widetilde{D}}\longrightarrow \mathcal F_2.$$
if $\mathcal F^j_1\subset \mathcal F_1$\, is a coherent subsheaf, then $d^N_{\mathcal F_1}(\mathcal F^j_1)\subset \mathcal J^N(\mathcal F_1/S)$\, is contained in the image of the canonical homomorphism $\mathcal J^N(\mathcal F^j_1/S)\longrightarrow \mathcal J^N(\mathcal F_1/S)$\,  since 
$$\mathcal F^j_1\longrightarrow  \mathcal J^N(O_X/S)\cdot d^N_{\mathcal F_1}(\mathcal F_1^j)$$ is a filtered module derivation  and the homomorphism  exists by the universal property of the jet-modules (see \prettyref{lem:L45}).  It follows that $D(\mathcal F^j_1)$\, is contained in the coherent subsheaf 
$$\mathcal F_2^j:=\widetilde{D}(\mathcal J^N(X/S)\cdot d^N_{\mathcal F_1}(\mathcal F^j_1))\subset \mathcal F_2.$$
Thus we can write $D$ as the filtered direct limit of differential operators $D_j: \mathcal F^j_1\longrightarrow \mathcal F^j_2$\, for some  directed index set $J$\,.
\end{proof}
\subsection{Basic properties of differential operators}
In this subsection, we collect the basic properties of linear algebraic differential operators which directly follow from the corresponding basic properties of the jet-modules.
\begin{lemma}\mylabel{lem:L295} Let $A\longrightarrow B$\, be homomorphisms of rings, $M_1,M_2$ be  $B$-modules $N$\, be an $A$-module, and $D: M_1\longrightarrow M_2$\, be a differential  operator relative to $A$. Then, $D\otimes_A\text{id}_N: M_1\otimes_AN\longrightarrow M_2\otimes_AN$\, is a differential operator of $B$-modules relative to $A$.
\end{lemma}
\begin{proof} This follows from the properties of the jet modules $\mathcal J^N(M_1\otimes_AN/A)\cong \mathcal J^N(M/A)\otimes_AN$\, (see \prettyref{prop:POkt2914}). The operator $D$ is given by a $B$-linear map $\widetilde{D}: \mathcal J^N(M_1/A)\longrightarrow M_2$\,, so we get a $B$-linear map 
$$\widetilde{D}\otimes_A\text{id}_N: \mathcal J^N(M_1/A)\otimes_AN=\mathcal J^N(M_1\otimes_AN/A),\longrightarrow M_2\otimes_AN$$ where the $B$-module structures are given by the first tensor factor.
\end{proof}
We can generalize the previous lemma slightly to
\begin{proposition}\mylabel{prop:P2910}(Base change 0) Let $ A\longrightarrow B$\, be homomorphisms of commutative rings, $D: M_1\longrightarrow M_2$\, be a differential operator between $B$-modules relative to $A$. Let $A\longrightarrow A'$\,
 be  homomorphism of commutative rings. Then $D\otimes_A\text{id}_{A'}: M_1\otimes_AA'\longrightarrow M_2\otimes_A{A'}$\, is a differential operator on $B\otimes_AA'$-modules relative to $A'$\,.
\end{proposition}
\begin{proof} This follows from the base change properties of the jet-modules : we have an isomorphism  
$$\mathcal J^N(M/A)\otimes_AA'\cong \mathcal J^N(M\otimes_AA'/A')$$ see \prettyref{prop:P2913} (Base change I for jet- modules).\\
 $D$ corresponds to a $B$-linear map $\widetilde{D}:\mathcal J^N(M_1/A)\longrightarrow M_2$\,. We get a $B\otimes_AA'$-linear map 
$$\widetilde{D}\otimes_A\mbox{Id}_{A'}: \mathcal J^N(M_1\otimes_AA'/A')\cong \mathcal J^N(M_1/A)\otimes_AA'\longrightarrow M_2\otimes_AA'$$ and the claim follows.
\end{proof}
\begin{proposition}\mylabel{prop:P19129}(Base change I) Let $A\longrightarrow B$\, and $A\longrightarrow A'$\, be homomorphisms of commutative rings and $N$\, be an $A'$-module. Given a differential operator $D: M_1\longrightarrow M_2$\, of $B$-modules relative to $A$, the $A'$-linear map
$$D\otimes_A\id_N: M_1\otimes_AN\longrightarrow M_2\otimes_AN$$\, is a differential operator of $B\otimes_AA'$-modules
\end{proposition}
\begin{proof} This follows in the standard way from the identity $\mathcal J^N(M\otimes_AN/A')\cong \mathcal J^N(M/A)\otimes_AN$\, (see \prettyref{prop:P2012}, Base change III).
\end{proof}
\begin{proposition}\mylabel{prop:POkt297}(exterior products) Let $k
\longrightarrow A$\, and $k\longrightarrow B$\, be homomorphisms of commutative rings, $D: M_1\longrightarrow M_2$\, and $E: N_1\longrightarrow N_2$\, be differential operators between $A$-modules $M_1,M_2$\, and $B$-modules $N_1,N_2$\, respectively. Then, the tensor product over $k$:
$$D\otimes_kE: M_1\otimes_kN_1\longrightarrow M_2\otimes_kN_2$$ 
is a differential operator between the $A\otimes_kB$-modules $M_1\otimes_kN_1$\, and $M_2\otimes_kN_2$\,.
\end{proposition}
\begin{proof} This follows from the properties of the jet-modules (exterior products II), namely 
$$\mathcal J^{\mathbb N}(M/k)\otimes_k\mathcal J^{\mathbb N}(N/k)\cong \mathcal J^{\mathbb N}(M\otimes_kN/k).$$ If for some $N\in \mathbb N_0 $\, there is an $A$-linear map $\widetilde{D}: \mathcal J^N(M_1/k)\longrightarrow M_1$\, and a $B$-linear map $\widetilde{E}: \mathcal J^N(N_1/k)\longrightarrow N_2$, for some large $N'>>N$\,, there is an $A\otimes_KB$-linear map
$$\mathcal J^{N'}(M_1\otimes_kN_1/k)\longrightarrow \mathcal J^N(M_1/k)\otimes_k\mathcal J^N(N_1/k)\stackrel{\widetilde{D}\otimes_k\widetilde{E}}\longrightarrow M_2\otimes_kN_2.$$
The proof in the opposite direction is the same, one has only to bear in mind that the order of the differential operators may change.
\end{proof} 
\begin{lemma}\mylabel{lem:L105} Let $k\longrightarrow A$\, be a homomorphism of rings and $M_1, M_2, N_1,N_2,\,$\, be $A$-modules with differential operators $D_i: M_i\longrightarrow N_i, i=1,2$\, and let $\phi_i: M_i\longrightarrow N_i, i=1,2$\, be $A$-linear maps commuting with the $D_i$\,.\\ Then, the $k$-linear maps induced by $D_i$, 
$$\text{ker}(\phi_1)\longrightarrow \text{ker}(\phi_2)\quad \text{and}\quad \text{coker}(\phi_1)\longrightarrow \text{coker}(\phi_2)$$ are differential operators relative to $k$. Also  an arbitrary direct sum  
$$\bigoplus_{i\in I}{D_i}: \bigoplus_{i\in I}M_i\longrightarrow \bigoplus_{i\in I}N_i$$ is a differential operator for an arbitrary index set $I$ and differential operators $D_i: M_i\longrightarrow N_i$\,.
\end{lemma}
\begin{proof} The last statement follows from the fact, that the jet-modules commute with direct sums. \\
For the first statement, observe, that, more generally, if $D: G\longrightarrow F$\, is a differential operator relative to $k$ and $G_1\subset G$\, and $F_1\subset F$\, are $A$-sub-modules with $D(G_1)\subset F_1$\,, then the restriction of $D$ to $G_1$\, is a differential operator relative to $k$\,.\\
Namely, consider the  $\mathcal J^N(A/k)$- module $\mathcal J^N(A/k)\cdot d^N_{G/k}(G_1)\subset \mathcal J^N(G/k)$\,.\\ Obviously, the restriction of $d^N_{G/k}$\, to $G_1$\, gives a module derivation 
$$t_{G_1/k}: G_1\longrightarrow \mathcal J^N(A/k)\cdot d^N_{G/k}(G_1). $$
By the universal property, ther exists a $\mathcal J^N(A/k)$-linear homomorphism 
$$\phi: \mathcal J^N(G_1/k)\longrightarrow \mathcal J^N(A/k)\cdot d^N_{G/k}(G_1),\quad \text{such that}\quad t_{G_1/k}=\phi\circ d^N_{G_1/S}.$$
If $\widetilde{D}: \mathcal J^N(G/k)\longrightarrow F$\, corresponds to $D$, I claim that 
$$\widetilde{D}(\mathcal J^N(A/k)\cdot d^N_{G/k}(G_1))\subset F_1.$$
Indeed, $\widetilde{D}$\, is $A$-linear with respect to the first $A$-module structure of $\mathcal J^N(G/k)$\,. The quantity $d^N_{G/k}(G_1)$\, is an $A$-submodule with respect to the second $A$-module structure of $\mathcal J^N(A/k)$\,. Thus, $\mathcal J^N(A/k)\cdot d^N_{G/k}(G_1)= A\cdot^{(1)}d^N_{G/k}(G_1)$
the superscript $(-)^{(1)}$\, indicates that we take the $A$-module generated by $d^N_{G/k}(G_1)$\, with respect to the first module structure. But this is then an $A$-bi-submodule, (since the bi-module structure is commutative), or, equivalently a $\mathcal J^N(A/k)$-submodule of $\mathcal J^N(G/k)$\,. Since $\widetilde{D}$\, is $A$-linear with respect to the first $A$-module structure, the claim follows. \\
Now, the statement for the cokernel.\\
Let $\pi_G: G\twoheadrightarrow G_2$\, and $\pi_F: F\twoheadrightarrow F_2$\, be $A$-linear surjective maps such that the given differential operator $D: G\longrightarrow F$\, 
descends to a $k$-linear map $D_2: G_2\longrightarrow F_2$\,.
 The claim is that $D_2$\, is a differential operator of $A$-modules relative to $k$\,. Let $G_1$\, and $F_1$\, be the kernels of $\pi_G$\, and $\pi_F$\, respectively,  with inclusions $i_G$\, and $i_F$\,.\\
By the first case, the induced map $D_1: G_1\longrightarrow F_1$\, is a differential operator relative to $k$. By the right exactness of the jet-module-functor we have 
\begin{gather*}\mathcal J^N(G_2/k)\cong \mathcal J^N(G/G_1/k)\cong \mathcal J^N(G/k)/\text{im}(\mathcal J^N(i_G/k)( \mathcal J^N(G_1/k)))
\end{gather*} We know, that $D_1$\, factors over a map 
$$\widetilde{D_1}: \text{im}(\mathcal J^N(i_G//k))(\mathcal J^N(G_1/k)))\longrightarrow F_1.$$ We can form the quotient map
$$\widetilde{D_2}: \mathcal J^N(G_2/k)\cong\mathcal J^N(G/k)/\text{im}(\mathcal J^N(i_G/k))(\mathcal J^N(G_1/k))\longrightarrow F/F_1\cong F_2.$$
Then, obviously, $D_2$\, factors over $\widetilde{D_2}$\, since the quotient map has to factor over the quotient module of the modules over which the first two maps factor. \\
This shows that $D_2: G_2\longrightarrow F_2$\, is a differential operator relative to $k$. The case of the map induced on cokernels is a special case of this. 
\end{proof}
\subsection{Existence of differential operators in the affine case}
We want to investigate the question, under which conditions $DO^N_{A/k}(M_1,M_2)\subsetneq DO^{N+1}_{A/k}(M_1,M_2)$\,, where $k\longrightarrow A$\, is a homomorphism of commutative rings $M_1,M_2$\, are $A$-modules. We start with a proposition. 
\begin{proposition}\mylabel{prop:P11171} Let $k\longrightarrow A$\, be a homomorphism of noetherian rings and $M_1,M_2$\, be finitely generated $A$-modules. 
\begin{enumerate}[1]
\item Let $\ann(M_2)=I$\,. Given any $A$-linear map
$\widetilde{D}: \mathcal J^N(M_1/k)\longrightarrow M_2$\, for some $N\in \mathbb N_0$\,, there is some $K\in \mathbb N$\, such that $\widetilde{D}$\, factors over $\widetilde{D'}: \mathcal J^N((M_1/I^K\cdot M_1)/k)\longrightarrow M_2$\,.
\item Let now $\ann(M_1)=I$\,. Then there is some $K\in \mathbb N$\, such that the image of $\widetilde{D}$\, is a submodule annihilated by $I^K$\,.
\end{enumerate}
\end{proposition}
\begin{proof} Denote by $I_{A/k}$\, the diagonal ideal, generated by all elements $(a\otimes 1-1\otimes a)$\,.
\begin{enumerate}[1]
\item Obviously the kernel of $\widetilde{D}$\, contains $I\cdot \mathcal J^N(M_1/k)=\overline{I\otimes_kM_1}$\,. We show that the ideal $\overline{I\otimes_kA}$\, contains $\overline{A\otimes_k I^K}$\, for some $K\in \mathbb N$\,.\\
We know that $I$ is finitely generated, $I=(f_1,...,f_k)$\,. In $\mathcal J^N(A/k)$\, we have $I_{A/k}^{N+1}=0$\, containing the elements 
$$(f_i\otimes 1-1\otimes f_i)^{N+1}= f_i\otimes \omega +1\otimes f_i^{N+1}\quad \text{and}$$ 
$1\otimes f_i^{N+1}\in \overline{I\otimes_kA}$\,. \\
It follows that there exists some $K\in \mathbb N$\, such that $\overline{I\otimes_kA}$\, contains $\overline{A\otimes_kI^K}$\, (one can take $K=(N+1)\cdot k$.)\, \\
Then, $\overline{I\otimes_kM}$\, contains $\overline{A\otimes_kM/I^K\cdot A\otimes M}$\,. Thus, the $A$-linear map $\widetilde{D}$\, factors through 
$$\overline{A\otimes M_1/ A\otimes I^K\cdot M_1}=\mathcal J^N((M_1/I^KM_1)/k).$$
\item By the same arguement, if $\ann(M_1)=I$\,, there is $K\in \mathbb N$\, such that $\ann(\mathcal J^N(M_1/k))$\, contains $\overline{I^K\otimes A}$\,.  Then, since  $\overline{A\otimes I}$\, is contained in the  annihilator of $\mathcal J^N(M_1/k)=\overline{A\otimes_kM}$, we have 
$$\overline{I^K\otimes A}\subseteq \overline{A\otimes I}\subset \ann(\mathcal J^N(M/k)).$$ Since $\widetilde{D}$\, is $A$-linear, the result follows.
\end{enumerate}
\end{proof}
We prove the following basic
\begin{theorem}\mylabel{thm:T101} Let $k\longrightarrow A$\, be a  homomorphism of finite type of noetherian rings  of characteristic zero of relative dimension (fibre dimension) $\geq 1$\, and $M_1,M_2$ be  finitely generated   $A$-modules.  Then, the inclusion 
$$DO^N_{A/k}(M_1,M_2)\subsetneq DO_{A/k}^{N+1}(M_1,M_2)$$ is for all $N\geq 0$\, strict in the following cases:
\begin{enumerate}[1]
\item $M_1=M_2$, or more  generally $\ann(M_1)=\ann(M_2)=I$\, and $\dim(\supp(M_1))$\, has fibre dimension greater than zero over $k$\,.
\item $M_1,M_2$\, are nontorsion modules on $A$.
\item  The module $M_1$\, is nontorsion and $M_2$\, is a torsion module with $V(\ann(M_2))=V(I)$\, having fibre dimension greater than zero over $k$.
\end{enumerate}
\end{theorem}
\begin{proof}
\begin{enumerate}[1]
\item The first case is easily reduced to the case (2). By \prettyref{lem:L10}, there is a surjection
$$DO^N_{A/k}(M_1,M_2)\twoheadrightarrow DO^N_{(A/I)/k}(M_1,M_2)$$
which induces a surjection
$$DO^N_{A/k}(M_1,M_2)/DO^{N-1}_{A/k}(M_1,M_2)\twoheadrightarrow DO^N_{(A/I)/k}(M_1,M_2)/DO^{N-1}_{(A/I)/k}(M_1,M_2).$$
\item  First, we treat for reasons of intuition the case where $A$\, is an integral $k$-algebra with $k$ being a field and $M_1,M_2$ being torsion free, or more generally nontorsion $A$-modules.\\  
If $M_i\cong A^{\oplus r_i}$\, is free, 
$$DO^N(M_1,M_2)=\Hom_A(\mathcal J^N(A/k)^{\oplus r_1}, A^{\oplus r_2})=DO^N(A,A)^{\oplus r_1\cdot r_2}$$ and we are reduced to the case $r_1=r_2=1$\,. We have the standard exact sequence 
$$j^N(A/k):0\longrightarrow  \mathcal I_{\Delta}^N/\mathcal I_{\Delta}^{N+1}\longrightarrow \mathcal J^N(A/k)\longrightarrow \mathcal J^{N-1}(A/k)\longrightarrow 0$$
 Let $X_{ns}\subset X=\Spec A$\, be the set of nonsingular points, which is a nonempty Zariski-open subset. Over $X_{ns}$\,, the $A$-modules $\mathcal J^N(A/k),\\
  \mathcal J^{N-1}(A/k)$\, are projective of different ranks,  since, considering the exact sequence $j^N(A/k)$\, we know that $I_{A/k}^N/I_{A/k}^{N+1}$\, is isomorphic to the $N^{th}$\, symmetric power of the relative cotangential sheaf. This is nonzero because the relative dimension of $A/k$\, was assumed to be greater than one. Thus, the  duals  $DO^N(A/k)$\, and $DO^{N-1}(A/k)$\, have over $X_{ns}$\, different ranks. It follows that the inclusion $DO^{N-1}(A/k)\subsetneq DO^N(A/k)$\,\, is strict, since the inclusion of the corresponding Zariski- sheaves on $\Spec A$ is strict.\\
This settles the free and projective case.\\
If $M_1$ and $M_2$ are any  nontorsion modules , there is an open subset $U\subset \Spec A$\, where $M_1$ and $M_2$\, are free and 
$DO_{U/k}^N(M_1,M_2)/DO_{U/k}^{N-1}(M_1,M_2)$\, is nonzero. Since \\
$DO^N_{A/k}(M_1,M_2)/DO^{N-1}_{A/k}(M_1,M_2)$\, is finitely generated and coherent, it follows that this module is nonzero.\\
We now treat the general nontorsion case. So let $k\longrightarrow A$\, be a homomorphism of noetherian rings which makes $A$ a finitely generated $k$-algebra and let $M_1,M_2$\, be nontorsion modules on $A$. Let $\eta_i, i=1,...,l$\, be the generic points of $A$. We want to show that the $A$-module\\ 
$DO^N_{A/k}(M_1,M_2)/DO^{N-1}_{A/k}(M_1,M_2)$\, is nonzero. To this aim, it suffices to show that this module is nonzero at the generic points of $A$, which reduces by the localization properties of the jet modules to the case, where $(A,\mathfrak{m}, \kappa)$\, is an Artinian $k$-algebra, essentially of finite type over $k$\, such that the resiude field $\kappa$\, has transcendence degree $\geq 1$\, over $k$.
First, we show that $l(\mathcal J^{N-1}(A/k))< l(\mathcal J^N(A/k))$\, for each $N\geq 1$\,, where $l(-)$\, denotes the length function. By the standard jet-module sequence and the additivity of the lenght function, we only have to show that the $A$-module $I_{A/k}^N/I_{A/k}^{N+1}$\, is nonzero. Let $A\twoheadrightarrow \kappa$\, be the canonical surjection. There is a surjection $I_{A/k}\twoheadrightarrow I_{\kappa/k}$\, which induces surjections
$$I_{A/k}^N/I_{A/k}^{N+1}\twoheadrightarrow I_{\kappa/k}^N/I_{\kappa/k}^{N+1}\quad \forall n\in \mathbb N.$$
The last module is equal to $\Omega^{(1)}(\kappa/k)^{\otimes^s N}$\, since we work in characteristic zero and this module is nonzero, since $\text{trdeg}(\kappa/k)$\, is greater or equal to one.  The case for general $M_1$\, is basically the same.
There is a surjection
$$I_{A/k}^N\cdot M/I_{A/k}^{N+1}\cdot M\twoheadrightarrow I_{\kappa/k}^N\cdot \overline{M}/I_{\kappa/k}^{N+1}\cdot \overline{M},$$ where $\overline{M}:=M\otimes_A\kappa$\,. Since $\kappa$\, is a field, $\overline{M}$\, is free, and the last quantity is isomorphic to $\Omega^{(1)}(\kappa/k)^{\otimes^s N}\otimes_{\kappa}\overline{M}$\, which is also nonzero.\\
Thus for each $A$-module $M$\, and each $N\in \mathbb N$\,, we have strict inequality
$l(\mathcal J^{N-1}(M/k)) <l(\mathcal J^N(M/k))$ coming from the exact sequence of nonzero $A$-modules
$$j^N_{A/k}(M): 0\longrightarrow I_{A/k}^N\cdot M/I_{A/k}^{N+1}\cdot M\longrightarrow \mathcal J^N(M/k)\longrightarrow \mathcal J^{N-1}(A/k)\longrightarrow 0.$$
By \cite{Eisenbud}[chapter 18, Proposition 18.4, p. 454], for given $A$-modules $N_1,N_2$\,, the minimal number $r$ such that $Ext^1_A(N_1,N_2)\neq 0$\, is given by  
$$r=\text{depth}(\ann(N_1),N_2).$$ The ideal $\ann(N_1)$ is contained in the maximal ideal, unless $N_1=0$\, which we want to exclude, and, the maximal ideal consists of zero divisors, since $A$ was assumed to be Artinian. Thus the depth is always equal to zero. Consequently, taking $\Hom_A(-,M_2)$\, of the exact sequence $j^N_{A/k}(M_1)$\,, we get an exact sequence
\begin{gather*}do^N_{A/k}(M_1,M_2): 0\longrightarrow DO^{N-1}_{A/k}(M_1,M_2)\longrightarrow DO^{N}_{A/k}(M_1,M_2)\\
\longrightarrow \Hom_A(I_{A/k}^N\cdot M_1/I_{A/k}^{N+1}\cdot M_1,M_2)\longrightarrow 0,
\end{gather*} where surjectivity on the right comes from the vanishing of  the $Ext^1$\, and the last $A$-module is nonzero by what has been just said. Thus, the quotient module
$$ DO^N_{A/k}(M_1,M_2)/DO^{N-1}_{A/k}(M_1,M_2)$$
is nonzero.
The general case where  $A/k$\, is an arbitrary finite type $k$-algebra follows from the fact, that the last quotient is a coherent sheaf on $\Spec A$\, which is nonzero at the generic points of $\Spec A$\, and thus nonzero.
\item Next, if $M_2$\, is a torsion module with $\ann(M_2)=I$, by the previous proposition , for each $N\in \mathbb N$\, there is $K_2=K(N)\in \mathbb N$\, such that $\widetilde{D}$\, factors over $\mathcal J^N((M_1/I^{K_2}\cdot M_1)/k)\longrightarrow M_2$\,.\\ 
 But then,  $D:M_1\longrightarrow M_2$\, factors through a differential operator $ D': M_1/I^{K_2}\cdot M_1\longrightarrow M_2.$
 Thus, we have shown that for fixed $M_1,M_2$\, with $\ann(M_2)=I$\, and $M_1$\, being a nontorsion module, for each $N\in \mathbb N$\, there is a surjection
 $$DO^N_{A/k}(M_1,M_2)\twoheadrightarrow DO^N_{(A/I^{K(N)}A)/k}((M_1/I^{K(N)}\cdot M_1), M_2).$$
  Observe, that if $D: M_1\longrightarrow M_2$\, factors over $M_1/I^K\cdot M_1$\,, then it certainly factors over $M_1/I^L\cdot M_1$\, for each $L\geq K$\,. Putting $K=\max(K(N),K(N-1))$\,, we get a surjection
 \begin{gather*}DO^N_{A/k}(M_1,M_2)/DO^{N-1}_{A/k}(M_1,M_2)\twoheadrightarrow \\
 DO^N_{(A/I^K)/k}((M_1/I^KM_1),M_2)/DO^{N-1}_{(A/I^K)/k}(M_1/I^KM_1, M_2).
 \end{gather*}
 Fixing $N\in \mathbb N_0$, we know by case (2), that the module on the left hand side is nonzero and the claim follows.
\end{enumerate}
\end{proof}
\begin{remark}\mylabel{rem:R14111} Observe, that, if $A/k$\, is a finitely generated Artinian $k$-algebra, there is an $N\in \mathbb N$\, such that $I_{\Delta}^N=0$\, in $A\otimes_kA$\, and so $\mathcal J^M(A/k)=A\otimes_kA$\, for $M\geq N$\,. In particular $DO^M(A,A)=\Hom_A(A\otimes_kA,A)$\, and the statement is not true. 
\end{remark} 
\bibliography{BasicsDiffops}
\bibliographystyle{plain}
\noindent
\emph{E-Mail-adress:}nplc@online.de
\end{document}